\pdfoutput=1

\documentclass[11pt]{amsart}%
\usepackage{amssymb}
\usepackage{amsthm,amsmath}
\usepackage{graphicx}
\usepackage{amsfonts}
\usepackage{amssymb,latexsym}
\usepackage{url}
\usepackage{amsmath}
\usepackage{algorithmic}
\usepackage[nothing]{algorithm}
\usepackage{float}%
\providecommand{\U}[1]{\protect\rule{.1in}{.1in}}
\theoremstyle{plain}
\newtheorem{theorem}{Theorem}
\newtheorem{corollary}[theorem]{Corollary}
\newtheorem{lemma}[theorem]{Lemma}
\newtheorem{proposition}[theorem]{Proposition}

\newtheorem{fact}[theorem]{Fact}
\theoremstyle{definition}

\theoremstyle{remark}

\newtheorem{remark}[theorem]{Remark}
\numberwithin{equation}{section}
\numberwithin{theorem}{section}

\begin{document}
\title{Axiomatizations and factorizations of\\Sugeno utility functions}
\author{Miguel Couceiro}
\address[Miguel Couceiro]{Mathematics Research Unit, FSTC, University of Luxembourg \\
6, rue Coudenhove-Kalergi, L-1359 Luxembourg, Luxembourg}
\email{miguel.couceiro[at]uni.lu }
\author{Tam\'as Waldhauser}
\address[Tam\'as Waldhauser]{Mathematics Research Unit, FSTC, University of Luxembourg \\
6, rue Coudenhove-Kalergi, L-1359 Luxembourg, Luxembourg, and Bolyai
Institute, University of Szeged, Aradi v\'{e}rtan\'{u}k tere 1, H-6720 Szeged, Hungary}
\email{twaldha@math.u-szeged.hu}
\maketitle

\begin{abstract}
In this paper we consider a multicriteria aggregation model where local
utility functions of different sorts are aggregated using Sugeno integrals,
and which we refer to as Sugeno utility functions. We propose a general
approach to study such functions via the notion of pseudo-Sugeno integral (or,
equivalently, pseudo-polynomial function), which naturally generalizes that of
Sugeno integral, and provide several axiomatizations for this class of functions.

Moreover, we address and solve the problem of factorizing a Sugeno utility
function as a composition $q(\varphi_{1}(x_{1}),\ldots,\varphi_{n}(x_{n}))$ of
a Sugeno integral $q$ with local utility functions $\varphi_{i}$, if such a
factorization exists.


\end{abstract}

\section{Introduction}

The importance of aggregation functions is made apparent by their wide use,
not only in pure mathematics (e.g., in the theory of functional equations,
measure and integration theory), but also in several applied fields such as
operations research, computer and information sciences, economics and social
sciences, as well as in other experimental areas of physics and natural
sciences. For general background, see~\cite{BelPraCal07,GraMurSug00} and for a
recent reference, see~\cite{GraMarMesPap09}.

In many applications, the values to be aggregated are first to be transformed
by mappings $\varphi_{i}\colon X_{i}\to Y$, $i=1, \ldots,n$, so that the
transformed values (which are usually real numbers) can be aggregated in a
meaningful way by a function $M\colon Y^{n}\to Y$. The resulting composed
function $U\colon X_{1}\times\cdots\times X_{n}\to Y$ is then defined by
\begin{equation}
\label{eq:OverPrefFun}U(x_{1},\ldots,x_{n})=M(\varphi_{1}(x_{1}),\ldots
,\varphi_{n}(x_{n})).
\end{equation}
Such an aggregation model is used for instance in multicriteria decision
making where the criteria are not commensurate. Here each $\varphi_{i}$ is a
local utility function, i.e., order-preserving mapping, and the resulting
function $U$ is referred to as an overall utility function (also called global
preference function). For general background see \cite{BouDubPraPir09}.

In this paper, we consider this aggregation model in a purely ordinal decision
setting, where $Y$ and each $X_{i}$ are bounded chains $L$ and $L_{i}$,
respectively, and where $M\colon L^{n}\to L$ is a Sugeno integral (see
\cite{DubMarPraRouSab01,Sug74,Sug77}) or, more generally, a lattice polynomial
function. We refer to the resulting compositions as pseudo-Sugeno integrals
and pseudo-polynomial functions, respectively. The particular case when each
$L_{i}$ is the same chain $L^{\prime}$, and each $\varphi_{i}$ is the same
mapping $\varphi\colon L^{\prime}\to L$, was studied in \cite{CouMar3} where
the corresponding compositions $U=M\circ\varphi$ were called quasi-Sugeno
integrals and quasi-polynomial functions. Such mappings were characterized as
solutions of certain functional equations and in terms of necessary and
sufficient conditions which have natural interpretations in decision making
and aggregation theory.

Here, we take a similar approach and study pseudo-Sugeno integrals from an
axiomatic point of view, and seek necessary and sufficient conditions for a
given function to be factorizable as a composition of a Sugeno integral with
unary maps. The importance of such an axiomatization is attested by the fact
that this framework subsumes the Sugeno utility model. Since overall utility
functions (\ref{eq:OverPrefFun}) where $M$ is a Sugeno integral, coincide
exactly with order-preserving pseudo-Sugeno integrals (see
Corollary~\ref{cor SUF=OP PSI}), we are particularly interested in the case
when the inner mappings $\varphi_{i}$ are local utility functions.

As mentioned, this aggregation model is deeply rooted in multicriteria
decision making, where the variables $x_{i}$ represent different properties of
the alternatives (e.g., price, speed, safety, comfort level of a car), and the
overall utility function assigns a score to the alternatives that helps the
decision maker to choose the best one (e.g., to choose the car to buy). A
similar situation is that of subjective evaluation (see \cite{BouDubPraPir09}%
): $f$ outputs the overall rating of a certain product by customers, and the
variables $x_{i}$ represent the various properties of that product. The way in
which these properties influence the overall rating can give information about
the attitude of the customers. A factorization of the (empirically) given
overall utility function $f$ in the form (\ref{eq:OverPrefFun}) can be used
for such an analysis; this is our main motivation to also address this problem.

The paper is organized as follows. In Section~\ref{Sec:2} we recall the basic
definitions and terminology, as well as the necessary results concerning
polynomial functions (and, in particular, Sugeno integrals) used in the
sequel. In Section~\ref{Sec:3}, we focus on pseudo-Sugeno integrals as a tool
to study certain overall utility functions. We introduce the notion of
pseudo-polynomial function in Subsection~\ref{pseudo} and show that, even
though seemingly more general, it can be equivalently defined in terms of
Sugeno integrals. An axiomatization of this class of generalized polynomial
functions is given in Subsection~\ref{carPseudo}. Sugeno utility functions are
introduced in Subsection~\ref{overall}, as certain order-preserving
pseudo-Sugeno integrals, and then characterized in Subsection~\ref{overall2}
by means of necessary and sufficient conditions which extend well-known
properties in aggregation function theory. Within this general setting for
studying Sugeno utility functions, it is natural to consider the inverse
problem which asks for factorizations of a Sugeno utility function as a
composition of a Sugeno integral with local utility functions. This question
is addressed in Section~\ref{sect construction}, where an algorithmic
procedure is provided for constructing these factorizations of Sugeno utility
functions. We present the algorithm in Subsection~\ref{subsect algo}, which is
illustrated by a concrete example in Subsection~\ref{subsect example}, and in
Subsection~\ref{subsect correct} we show that this procedure does indeed
produce the desired factorizations.

This manuscript is an extended version of \cite{CW1} and \cite{CW2}, whose
results were presented at the conference MDAI 2010.

\section{Lattice polynomial functions and Sugeno integrals}

\label{Sec:2}

\subsection{Preliminaries}

Throughout this paper, let $L$ be an arbitrary bound\-ed chain endowed with
lattice operations $\wedge$ and $\vee$, and with least and greatest elements
$0_{L}$ and $1_{L}$, respectively; the subscripts may be omitted when the
underlying lattice is clear from the context. A subset $S$ of a chain $L$ is
said to be \emph{convex} if for every $a,b\in S$ and every $c\in L$ such that
$a\leq c\leq b$, we have $c\in S$. For any subset $S\subseteq L$, we denote by
$\mathop{\rm cl}\nolimits(S)$ the convex hull of $S$, that is, the smallest
convex subset of $L$ containing $S$. For instance, if $a,b\in L$ such that
$a\leq b$, then $\mathop{\rm cl}\nolimits(\{a,b\})=[a,b]=\{c\in L:a\leq c\leq
b\}.$

For an integer $n\geq1$, we set $[n]=\{1,\ldots,n\}$. Let $\sigma$ be a
permutation on $[n]$. The \emph{standard simplex} of $L^{n}$ associated with
$\sigma$ is the subset $L_{\sigma}^{n}\subseteq L^{n}$ defined by
\[
L_{\sigma}^{n}=\{\mathbf{x}\in L^{n}\colon x_{\sigma(1)}\leq x_{\sigma(2)}%
\leq\cdots\leq x_{\sigma(n)}\}.
\]
Two tuples are said to be \emph{comonotonic}, if there is a standard simplex
containing both of them.

Given arbitrary bounded chains $L_{i}$, $i\in\lbrack n]$, their Cartesian
product $\prod_{i\in\lbrack n]}L_{i}$ constitutes a bounded distributive
lattice by defining
\[
\mathbf{a}\wedge\mathbf{b}=(a_{1}\wedge b_{1},\ldots,a_{n}\wedge b_{n}%
),\quad\mbox{and}\quad\mathbf{a}\vee\mathbf{b}=(a_{1}\vee b_{1},\ldots
,a_{n}\vee b_{n}).
\]
For $k\in\left[  n\right]  $ and $c\in L_{k}$, we use $\mathbf{x}_{k}^{c}$ to
denote the tuple whose $i$th component is $c$, if $i=k$, and $x_{i}$, otherwise.

For $c\in L$ and $\mathbf{x}\in L^{n}$, let $\mathbf{x}\wedge c=(x_{1}\wedge
c,\ldots,x_{n}\wedge c)$ and $\mathbf{x}\vee c=(x_{1}\vee c,\ldots,x_{n}\vee
c),$ and denote by $[\mathbf{x}]_{c}$ the $n$-tuple whose $i$th component is
$0$, if $x_{i}\leq c$, and $x_{i}$, otherwise, and by $[\mathbf{x}]^{c}$ the
$n$-tuple whose $i$th component is $1$, if $x_{i}\geq c$, and $x_{i}$, otherwise.

Let $f\colon\prod_{i\in\lbrack n]}L_{i}\rightarrow L$ be a function. The
\emph{range} of $f$ is given by ${\mathrm{ran}}(f)=\{f(\mathbf{x}%
):\mathbf{x}\in\prod_{i\in\lbrack n]}L_{i}\}$. Also, $f$ is said to be
\emph{order-preserving} if, for every $\mathbf{a},\mathbf{b}\in\prod
_{i\in\lbrack n]}L_{i}$ such that $\mathbf{a}\leq\mathbf{b}$, we have
$f(\mathbf{a})\leq f(\mathbf{b})$. A well-known example of an order-preserving
function is the \emph{median} function $\mathop{\rm med}\nolimits\colon
L^{3}\rightarrow L$ given by
\[
\mathop{\rm med}\nolimits(x_{1},x_{2},x_{3})=(x_{1}\wedge x_{2})\vee
(x_{1}\wedge x_{3})\vee(x_{2}\wedge x_{3}).
\]
Given a tuple $\mathbf{x}\in L^{m}$, $m\geq1$, set $\langle\mathbf{x}%
\rangle_{f}=\mathop{\rm med}\nolimits(f(\mathbf{0}),\mathbf{x},f(\mathbf{1}))$.

\subsection{Basic background on polynomial functions and Sugeno integrals}

In this subsection we recall some well-known results concerning polynomial
functions that will be needed hereinafter. For further background, we refer
the reader to \cite{BurSan81,CouMar0,CouMar1,CouMar2,Goo67,Grae03,Rud01}.

Recall that a (\emph{lattice}) \emph{polynomial function} on $L$ is any map
$p\colon L^{n}\to L$ which can be obtained as a composition of the lattice
operations $\wedge$ and $\vee$, the projections $\mathbf{x}\mapsto x_{i}$ and
the constant functions $\mathbf{x}\mapsto c$, $c\in L$.

Polynomial functions are known to generalize certain prominent fuzzy
integrals, namely, the so-called (\emph{discrete}) \emph{Sugeno integrals}.
Indeed, as observed in \cite{Mar00,Mar09}, Sugeno integrals coincide exactly
with those polynomial functions $q:L^{n}\rightarrow L$ which are
\emph{idempotent}, that is, which satisfy $q(c,\ldots,c)=c$, for every $c\in
L$. In particular we have ${\mathrm{ran}}(q)=L$. We shall take this as our
working definition of the Sugeno integral; for the original definition (as an
integral with respect to a fuzzy measure) see, e.g.,
\cite{GraMarMesPap09,Sug74,Sug77}.

As shown by Goodstein \cite{Goo67}, polynomial functions over bounded
distributive lattices (in particular, over bounded chains) have very neat
normal form representations. For $I\subseteq\lbrack n]$, let $\mathbf{e}_{I}$
be the \emph{characteristic vector} of $I$, i.e., the $n$-tuple in $L^{n}$
whose $i$th component is $1$ if $i\in I$, and 0 otherwise.

\begin{theorem}
[Goodstein \cite{Goo67}]\label{prop:DNF(f)} A function $p\colon L^{n}%
\rightarrow L$ is a polynomial function if and only if
\begin{equation}
p(x_{1},\ldots,x_{n})=\bigvee_{I\subseteq\lbrack n]}\big(p(\mathbf{e}%
_{I})\wedge\bigwedge_{i\in I}x_{i}\big). \label{eq:Good}%
\end{equation}
Furthermore, the function given by \textup{(\ref{eq:Good})} is a Sugeno
integral if and only if $p(\mathbf{0})=0$ and $p(\mathbf{1})=1$.
\end{theorem}

\begin{remark}
\label{rem:Good} Observe that, by Theorem~\ref{prop:DNF(f)}, every polynomial
function $p\colon L^{n}\rightarrow L$ is uniquely determined by its
restriction to $\{0,1\}^{n}$. Also, since every lattice polynomial function is
order-preserving, we have that the coefficients in (\ref{eq:Good}) are
monotone increasing, i.e., $p(\mathbf{e}_{I})\leq p(\mathbf{e}_{J})$ whenever
$I\subseteq J$. Moreover, a function $f\colon\{0,1\}^{n}\rightarrow L$ can be
extended to a polynomial function over $L$ if and only if it is order-preserving.
\end{remark}

\begin{remark}
It follows from Goodstein's theorem that every unary polynomial function is of
the form
\begin{equation}
p(x)=s\vee(x\wedge t)=\mathop{\rm med}\nolimits(s,x,t)=\left\{
\begin{array}
[c]{rl}%
s,~ & \mbox{if }x<s,\\
x,~ & \mbox{if }x\in\lbrack s,t],\\
t,~ & \mbox{if }t<x,
\end{array}
\right.  \label{eq unary pol}%
\end{equation}
where $s=p(0),t=p(1)$. In other words, $p(x)$ is a truncated identity
function. Figure~\ref{fig1} shows the graph of this function in the case when
$L$ is the real unit interval $[0,1]$.
\end{remark}

%

\begin{figure}
[ptb]
\begin{center}
\includegraphics[
natheight=5.258300cm,
natwidth=5.362000cm,
height=5.2313cm,
width=5.3349cm
]%
{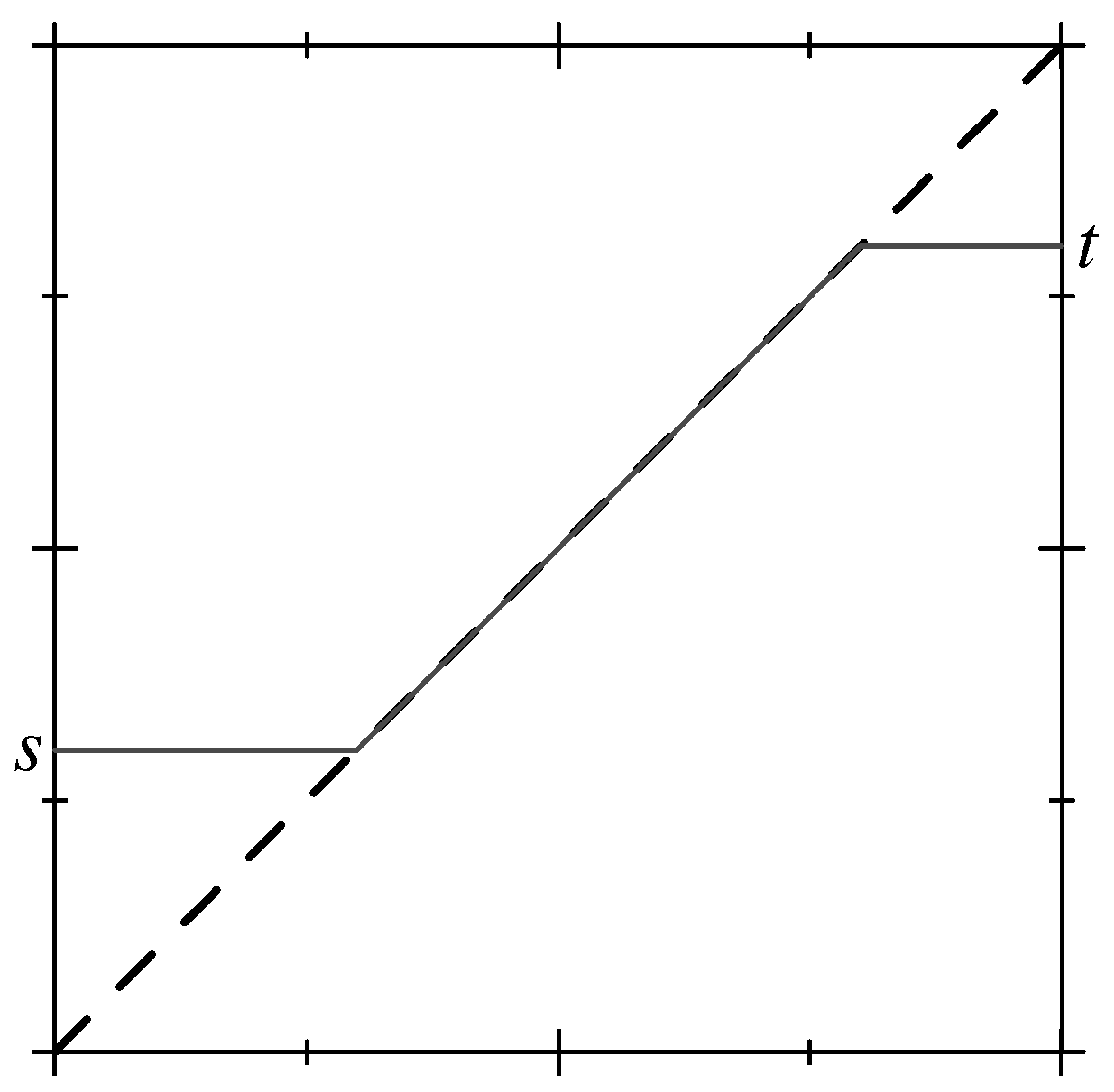}%
\caption{A typical unary polynomial function}%
\label{fig1}%
\end{center}
\end{figure}

It is noteworthy that every polynomial function $p$ as in (\ref{eq:Good}) can
be represented by $p=\langle q\rangle_{p}$ where $q$ is the Sugeno integral
given by
\begin{align*}
q(x_{1}, \ldots,x_{n}) = \bigvee_{\emptyset\subsetneq I \subsetneq[n]}
\big( p(\mathbf{e}_{I}) \wedge\bigwedge_{i \in I} x_{i} \big)\vee\bigwedge_{i
\in[n]} x_{i}.
\end{align*}

\subsection{Characterizations of polynomial functions}

\label{subsec:Pol}

The following results reassemble the various characterizations of polynomial
functions obtained in \cite{CouMar0}. For further background see, e.g.,
\cite{CouMar1,CouMar2,GraMarMesPap09}.

\begin{theorem}
\label{mainChar} Let $p\colon L^{n}\rightarrow L$ be a function on an
arbitrary bounded chain $L$. The following conditions are equivalent:

\begin{enumerate}
\item[$(i)$] $p$ is a polynomial function.

\item[$(ii)$] $p$ is median decomposable, that is, for every $\mathbf{x}\in
L^{n}$,
\[
p(\mathbf{x})=\mathop{\rm med}\nolimits\big(p(\mathbf{x}_{k}^{0}%
),x_{k},p(\mathbf{x}_{k}^{1})\big)\qquad(k=1,\ldots,n).
\]

\item[$(iii)$] $p$ is order-preserving, and
$\mathop{\rm cl}\nolimits({\mathrm{ran}}(p))$-min and
$\mathop{\rm cl}\nolimits({\mathrm{ran}}(p))$-max homogeneous, that is, for
every $\mathbf{x}\in L^{n}$ and every $c\in
\mathop{\rm cl}\nolimits({\mathrm{ran}}(p))$,
\[
p(\mathbf{x}\wedge c)=p(\mathbf{x})\wedge c\quad\mbox{and}\quad p(\mathbf{x}%
\vee c)=p(\mathbf{x})\vee c,\quad\mbox{resp.}
\]

\item[$(iv)$] $p$ is order-preserving, range-idempotent, and horizontally
minitive and maxitive, that is, for every $\mathbf{x}\in L^{n}$ and every
$c\in L$,
\[
p(\mathbf{x})=p(\mathbf{x}\vee c)\wedge p([\mathbf{x}]^{c})\quad
\mbox{and}\quad p(\mathbf{x})=p(\mathbf{x}\wedge c)\vee p([\mathbf{x}%
]_{c}),\quad\mbox{resp.}
\]

\end{enumerate}
\end{theorem}

\begin{remark}
Note that, by the equivalence $(i)\Leftrightarrow(iii)$, for every polynomial
function $p\colon L^{n}\rightarrow L$, $p(\mathbf{x})=\langle p(\mathbf{x}%
)\rangle_{p}= p(\langle\mathbf{x}\rangle_{p})$. Moreover, for every function
$f\colon L^{m}\rightarrow L$ and every Sugeno integral $q\colon L^{n}%
\rightarrow L$, we have $\langle q(\mathbf{x})\rangle_{f}= q(\langle
\mathbf{x}\rangle_{f})$.
\end{remark}

Theorem~\ref{mainChar} is a refinement of the Main Theorem in \cite{CouMar0}
originally stated for functions over bounded distributive lattices. As shown
in \cite{CouMar2}, in the case when $L$ is a chain, Theorem~\ref{mainChar} can
be strengthened since the conditions need to be verified only on tuples of
certain prescribed types. Moreover, further characterizations are available
and given in terms of conditions of somewhat different flavor, as the
following theorem illustrates \cite{CouMar2}.

\begin{theorem}
\label{theorem:WLP-comonot} A function $p\colon L^{n}\rightarrow L$ is a
polynomial function if and only if it is range-idempotent, and comonotonic
minitive and maxitive, that is, for any two comonotonic tuples $\mathbf{x}$
and $\mathbf{x}^{\prime}$ we have%
\[
p(\mathbf{x}\wedge\mathbf{x}^{\prime})=p(\mathbf{x})\wedge p(\mathbf{x}%
^{\prime})\quad\mbox{and}\quad p(\mathbf{x}\vee\mathbf{x}^{\prime
})=p(\mathbf{x})\vee p(\mathbf{x}^{\prime}),\quad\mbox{respectively.}
\]

\end{theorem}


\section{Pseudo-Sugeno integrals and Sugeno utility functions}

\label{Sec:3}

In this section we study certain prominent function classes in the realm of
multicriteria decision making. More precisely, we investigate overall utility
functions $U\colon\prod_{i\in\lbrack n]}L_{i}\rightarrow L$ which can be
obtained by aggregating various local utility functions (i.e.,
order-preserving mappings) $\varphi_{i}\colon L_{i}\rightarrow L$,
$i\in\lbrack n]$, using Sugeno integrals.

To this extent, in Subsection~\ref{pseudo} we introduce the wider class of
pseudo-polyno\-mial functions, and we present their axiomatization in
Subsection~\ref{carPseudo}. As we will see, pseudo-polynomial functions can be
equivalently defined in terms of Sugeno integrals, and thus they model certain
processes within multicriteria decision making. This is observed in Subsection
\ref{overall} where the notion of a Sugeno utility function $U\colon
\prod_{i\in\lbrack n]}L_{i}\rightarrow L$ associated with given local utility
functions $\varphi_{i}\colon L_{i}\rightarrow L$, $i\in\lbrack n]$, is
discussed. Using the axiomatization of pseudo-polynomial functions, in
Subsection~\ref{overall2} we establish several characterizations of Sugeno
utility functions given in terms of necessary and sufficient conditions which
naturally extend those presented in Subsection~\ref{subsec:Pol}.

\subsection{Pseudo-Sugeno integrals and pseudo-polynomial functions}

\label{pseudo} Let $L$ and $L_{1},\ldots,L_{n}$ be bounded chains. We shall
denote the top and bottom elements of $L_{1},\ldots,L_{n}$ and $L$ by $1$ and
$0$, respectively. This convention will not give rise to ambiguities. We shall
say that a mapping $\varphi_{i}\colon L_{i}\rightarrow L$, $i\in\lbrack n]$,
satisfies the \emph{boundary conditions} if for every $x\in L_{i}$,
\[
\varphi_{i}(0)\leq\varphi_{i}(x)\leq\varphi_{i}(1)\quad\mbox{or}\quad
\varphi_{i}(1)\leq\varphi_{i}(x)\leq\varphi_{i}(0).
\]
Observe that if $\varphi_{i}$ is order-preserving, then it satisfies the
boundary conditions. To simplify our exposition, we will assume that
$\varphi_{i}(0)\leq\varphi_{i}(x)\leq\varphi_{i}(1)$ holds; this can be always
achieved by replacing $L_{i}$ by its dual if necessary.

A function $f\colon\prod_{i\in[n]}L_{i}\to L$ is a \emph{pseudo-polynomial
function} if there is a polynomial function $p\colon L^{n} \to L$ and there
are unary functions $\varphi_{i}\colon L_{i}\to L$, $i\in[n]$, satisfying the
boundary conditions, such that
\begin{equation}
\label{eq:generalPol}f(\mathbf{x})=p(\varphi_{1}(x_{1}),\ldots,\varphi
_{n}(x_{n})).
\end{equation}
If $p$ is a Sugeno integral, then we say that $f$ is a \emph{pseudo-Sugeno
integral}. As the following result asserts, the notions of pseudo-polynomial
function and pseudo-Sugeno integral turn out to be equivalent.

\begin{proposition}
\label{prop:PseudoPol-Sugeno} A function $f\colon\prod_{i\in[n]}L_{i}\to L$ is
a pseudo-polynomial function if and only if it is a pseudo-Sugeno integral.
\end{proposition}

\begin{proof}
Clearly, every pseudo-Sugeno integral is a pseudo-polynomial function.
Conversely, if $f\colon\prod_{i\in\lbrack n]}L_{i}\rightarrow L$ is a function
of the form $f\left(  \mathbf{x}\right)  =p(\varphi_{1}(x_{1}),\ldots
,\varphi_{n}(x_{n}))$ for a lattice polynomial $p$, then by setting $\phi
_{i}=\langle\varphi_{i}\rangle_{p}$ and taking $q$ as a Sugeno integral such
that $p=\langle q\rangle_{p}$, we have
\begin{align*}
f(\mathbf{x}) &  =\langle q(\varphi_{1}(x_{1}),\ldots,\varphi_{n}%
(x_{n}))\rangle_{p}~=~q(\langle\varphi_{1}(x_{1})\rangle_{p},\ldots
,\langle\varphi_{n}(x_{n})\rangle_{p})\\
&  =q(\phi_{1}(x_{1}),\ldots,\phi_{n}(x_{n})),
\end{align*}
and thus $f$ is a pseudo-Sugeno integral.
\end{proof}

\begin{remark}
\label{rem:boundary} Clearly, if $f\left(  \mathbf{x}\right)  =p(\varphi
_{1}\left(  x_{1}\right)  ,\ldots,\varphi_{n}\left(  x_{n}\right)  )$ is a
pseudo-polynomial function, then for all $k\in\left[  n\right]  $ and
$\mathbf{x}\in\prod_{i\in\lbrack n]}L_{i}$ we have
\begin{equation}
f(\mathbf{x}_{k}^{0})\leq f(\mathbf{x})\leq f(\mathbf{x}_{k}^{1}).
\label{eq BCn}%
\end{equation}

\end{remark}

\subsection{A characterization of pseudo-Sugeno integrals}

\label{carPseudo}

Throughout this subsection, we assume that the unary maps $\varphi_{i}\colon
L_{i}\rightarrow L$ considered, satisfy the boundary condition $\varphi
_{i}(0)\leq\varphi_{i}(x)\leq\varphi_{i}(1)$.

We say that $f\colon\prod_{i\in\lbrack n]}L_{i}\rightarrow L$ is
\emph{pseudo-median decomposable} if for each $k\in\lbrack n]$ there is a
unary function $\varphi_{k}\colon L_{k}\rightarrow L$ such that
\begin{equation}
f(\mathbf{x})=\mathop{\rm med}\nolimits\big(f(\mathbf{x}_{k}^{0}),\varphi
_{k}(x_{k}),f(\mathbf{x}_{k}^{1})\big) \label{medianDecomp}%
\end{equation}
for every $\mathbf{x}\in\prod_{i\in\lbrack n]}L_{i}$. Note that if $f$ is
pseudo-median decomposable w.r.t. unary functions $\varphi_{i}\colon
L_{i}\rightarrow L$, $i\in\lbrack n]$, then (\ref{eq BCn}) holds.

\begin{theorem}
\label{thm:Med} Let $f\colon\prod_{i\in[n]}L_{i}\to L$ be a function. Then $f$
is a pseudo-Sugeno integral if and only if $f$ is pseudo-median decomposable.
\end{theorem}

\begin{proof}
First we show that the condition is necessary. Suppose that the function
$f\colon\prod_{i\in\lbrack n]}L_{i}\rightarrow L$ is of the form
$f(\mathbf{x})=q(\varphi_{1}(x_{1}),\ldots,\varphi_{n}(x_{n}))$ for some
Sugeno integral $q$ and unary functions $\varphi_{k}$ satisfying the boundary
conditions. We prove (\ref{medianDecomp}) for $k=1$; the other cases can be
dealt with similarly. Let us fix the values of $x_{2},\ldots,x_{n}$, and let
us consider the unary polynomial function $u\left(  y\right)  =q\left(
y,\varphi_{2}\left(  x_{2}\right)  ,\ldots,\varphi_{n}\left(  x_{n}\right)
\right)  $.

Setting $a=\varphi_{1}\left(  0\right)  ,b=\varphi_{1}\left(  1\right)
,y_{1}=\varphi_{1}\left(  x_{1}\right)  $, the equality to prove takes the
form $u\left(  y_{1}\right)  =\mathop{\rm med}\nolimits\left(  u\left(
a\right)  ,y_{1},u\left(  b\right)  \right)  $. This becomes clear if we take
into account that $u$ is of the form (\ref{eq unary pol}), and by the boundary
condition $a\leq y_{1}\leq b$ (see also Figure~\ref{fig1}).

To verify that the condition is sufficient, just observe that applying
(\ref{medianDecomp}) repeatedly to each variable of $f$ we can
straightforwardly obtain a representation of $f$ as $f(\mathbf{x}%
)=p(\varphi_{1}(x_{1}),\ldots,\varphi_{n}(x_{n}))$ for some polynomial
function $p$. Thus, $f$ is a pseudo-polynomial function and, by
Proposition~\ref{prop:PseudoPol-Sugeno}, it is a pseudo-Sugeno integral.
\end{proof}

In the next theorem we give a disjunctive normal form of the polynomial $p$
obtained at the end of the proof of the above theorem (by repeated
applications of the pseudo-median decomposition formula). Here $\mathbf{e}%
_{I}$ denotes the characteristic vector of $I\subseteq\lbrack n]$ in
$\prod_{i\in\lbrack n]}L_{i}$, i.e., the $n$-tuple in $\prod_{i\in\lbrack
n]}L_{i}$ whose $i$-th component is $1_{L_{i}}$ if $i\in I$, and $0_{L_{i}}$ otherwise.

\begin{theorem}
\label{thm dnf for pseudo-Sugeno}If $f\colon\prod_{i\in\lbrack n]}%
L_{i}\rightarrow L$ is pseudo-median decomposable w.r.t. unary functions
$\varphi_{i}\colon L_{i}\rightarrow L$, $i\in\lbrack n]$, then $f\left(
\mathbf{x}\right)  =p(\varphi_{1}\left(  x_{1}\right)  ,\ldots,\varphi
_{n}\left(  x_{n}\right)  )$, where $p$ is given by%
\[
p\left(  x_{1},\ldots,x_{n}\right)  =\bigvee\limits_{I\subseteq\left[
n\right]  }\bigl(f\left(  \mathbf{e}_{I}\right)  \wedge\bigwedge\limits_{i\in
I}x_{i}\bigr).
\]

\end{theorem}

\begin{proof}
We need to prove that the following identity holds:
\begin{equation}
f\left(  x_{1},\ldots,x_{n}\right)  =\bigvee\limits_{I\subseteq\left[
n\right]  }\bigl(f\left(  \mathbf{e}_{I}\right)  \wedge\bigwedge\limits_{i\in
I}\varphi_{i}\left(  x_{i}\right)  \bigr). \label{eq dnf for overall utility}%
\end{equation}
We proceed by induction on $n$. If $n=1$, then the right hand side of
(\ref{eq dnf for overall utility}) takes the form $f\left(  0\right)
\vee\left(  f\left(  1\right)  \wedge\varphi_{1}\left(  x_{1}\right)  \right)
=\operatorname{med}\left(  f\left(  0\right)  ,\varphi_{1}\left(
x_{1}\right)  ,f\left(  1\right)  \right)  $, which equals $f\left(
x_{1}\right)  $ by (\ref{medianDecomp}). Now suppose that the statement of the
theorem is true for all pseudo-median decomposable functions in $n-1$
variables. Applying the pseudo-median decomposition to $f$ with $k=n$ we
obtain
\begin{align}
f\left(  x_{1},\ldots,x_{n}\right)   &  =\operatorname{med}\left(
f_{0}\left(  x_{1},\ldots,x_{n-1}\right)  ,\varphi_{n}\left(  x_{n}\right)
,f_{1}\left(  x_{1},\ldots,x_{n-1}\right)  \right)
\label{eq med in nth variable}\\
&  =f_{0}\left(  x_{1},\ldots,x_{n-1}\right)  \vee\left(  f_{1}\left(
x_{1},\ldots,x_{n-1}\right)  \wedge\varphi_{n}\left(  x_{n}\right)  \right)
,\nonumber
\end{align}
where $f_{0}$ and $f_{1}$ are the $(n-1)$-ary functions defined by%
\begin{align*}
f_{0}\left(  x_{1},\ldots,x_{n-1}\right)   &  =f\left(  x_{1},\ldots
,x_{n-1},0\right)  ,\\
f_{1}\left(  x_{1},\ldots,x_{n-1}\right)   &  =f\left(  x_{1},\ldots
,x_{n-1},1\right)  .
\end{align*}
It is easy to verify that $f_{0}$ and $f_{1}$ are pseudo-median decomposable
w.r.t. $\varphi_{1},\ldots,\varphi_{n-1}$, therefore we can apply the
induction hypothesis to these functions:%
\begin{align*}
f_{0}\left(  x_{1},\ldots,x_{n-1}\right)  =  &  \bigvee\limits_{I\subseteq
\left[  n-1\right]  }\bigl(f_{0}\left(  \mathbf{e}_{I}\right)  \wedge
\bigwedge\limits_{i\in I}\varphi_{i}\left(  x_{i}\right)  \bigr)=\\
&  \bigvee\limits_{I\subseteq\left[  n-1\right]  }\bigl(f\left(
\mathbf{e}_{I}\right)  \wedge\bigwedge\limits_{i\in I}\varphi_{i}\left(
x_{i}\right)  \bigr),\\
f_{1}\left(  x_{1},\ldots,x_{n-1}\right)  =  &  \bigvee\limits_{I\subseteq
\left[  n-1\right]  }\bigl(f_{1}\left(  \mathbf{e}_{I}\right)  \wedge
\bigwedge\limits_{i\in I}\varphi_{i}\left(  x_{i}\right)  \bigr)=\\
&  \bigvee\limits_{I\subseteq\left[  n-1\right]  }\bigl(f\left(
\mathbf{e}_{I\cup\left\{  n\right\}  }\right)  \wedge\bigwedge\limits_{i\in
I}\varphi_{i}\left(  x_{i}\right)  \bigr).
\end{align*}
Substituting back into (\ref{eq med in nth variable}) and using distributivity
we obtain the desired equality (\ref{eq dnf for overall utility}).
\end{proof}

Let us note that the polynomial $p$ given in the above theorem is a Sugeno
integral if and only if $f\left(  \mathbf{0}\right)  =0$ and $f\left(
\mathbf{1}\right)  =1$.

\subsection{Motivation: overall utility functions}

\label{overall}

Despite the theoretical interest, the motivation for the study of
pseudo-Sugeno integrals (or, equivalently, pseudo-polynomial functions) is
deeply rooted in multicriteria decision making. Let $\varphi_{i}\colon
L_{i}\rightarrow L$, $i\in\lbrack n]$, be local utility functions (i.e.,
order-preserving mappings) having a common range ${{\mathcal{R}}}\subseteq L$,
and let $M\colon L^{n}\rightarrow L$ be an aggregation function. The
\emph{overall utility function} associated with $\varphi_{i}$, $i\in\lbrack
n]$, and $M$ is the mapping $U\colon\prod_{i\in\lbrack n]}L_{i}\rightarrow L$
defined by
\begin{equation}
U(\mathbf{x})=M(\varphi_{1}(x_{1}),\ldots,\varphi_{n}(x_{n})).
\label{eq:Utility}%
\end{equation}
For background on overall utility functions, see e.g.
\cite{BouDubPraPir09,Gra}.

Thus, pseudo-Sugeno integrals subsume those overall utility functions
(\ref{eq:Utility}) where the aggregation function $M$ is a Sugeno integral. In
the sequel we shall refer to a mapping $f\colon\prod_{i\in\lbrack n]}%
L_{i}\rightarrow L$ for which there are local utility functions $\varphi_{i}$,
$i\in\lbrack n]$, and a Sugeno integral (or, equivalently, a polynomial
function) $q$, such that
\begin{equation}
f(\mathbf{x})=q(\varphi_{1}(x_{1}),\ldots,\varphi_{n}(x_{n})),
\label{eq:SugenoUtility}%
\end{equation}
as a \emph{Sugeno utility function}. As it will become clear in
Corollary~\ref{cor SUF=OP PSI}, these Sugeno utility functions coincide
exactly with those pseudo-Sugeno integrals (or equivalently, pseudo-polynomial
functions) which are order-preserving. Also, by taking $L_{1}=\cdots=L_{n}=L$
and $\varphi_{1}=\cdots=\varphi_{n}=\varphi$, it follows that Sugeno utility
functions subsume the notions of quasi-Sugeno integral and quasi-polynomial
function in the terminology of \cite{CouMar3}.

\begin{remark}
\label{remark3} Note that the condition that $\varphi_{i}\colon L_{i}%
\rightarrow L$, $i\in\lbrack n]$ have a common range ${\mathcal{R}}$ is not
really restrictive, since each $\varphi_{i}$ can be extended to a local
utility function $\varphi_{i}^{\prime}\colon L_{i}^{\prime}\rightarrow L$,
where $L_{i}\subseteq L_{i}^{\prime}$, in such a way that each $\varphi
_{i}^{\prime}$, $i\in\lbrack n]$, has the same range ${\mathcal{R}}\subseteq
L$. In fact, if ${\mathcal{R}}_{i}$ is the range of $\varphi_{i}$, for each
$i\in\lbrack n]$, then ${\mathcal{R}}$ can be chosen as the interval
\[
\mathop{\rm cl}\nolimits(\bigcup_{i\in\lbrack n]}{\mathcal{R}}_{i}%
)=[\bigwedge_{i\in\lbrack n]}\varphi_{i}(0),\bigvee_{i\in\lbrack n]}%
\varphi_{i}(1)].
\]
In this way, if $f^{\prime}\colon\prod_{i\in\lbrack n]}L_{i}^{\prime
}\rightarrow L$ is such that $f^{\prime}(\mathbf{x})=q(\varphi_{1}^{\prime
}(x_{1}),\ldots,\varphi_{n}^{\prime}(x_{n}))$, then the restriction of
$f^{\prime}$ to ${\prod_{i\in\lbrack n]}L_{i}}$ coincides with the function
$f(\mathbf{x})=q(\varphi_{1}(x_{1}),\ldots,\varphi_{n}(x_{n}))$.
\end{remark}


\subsection{Characterizations of Sugeno utility functions}

\label{overall2}

In view of the remark above, in this subsection we will assume that the local
utility functions $\varphi_{i}\colon L_{i}\to L$, $i\in[n]$, considered have
the same range ${\mathcal{R}}\subseteq L$. Since local utility functions
satisfy the boundary conditions, from Theorem~\ref{thm:Med} we get the
following characterization of Sugeno utility functions.

\begin{corollary}
\label{Cor:Med} A function $f\colon\prod_{i\in[n]}L_{i}\to L$ is a Sugeno
utility function if and only if it is pseudo-median decomposable w.r.t. local
utility functions.
\end{corollary}

We will provide further axiomatizations of Sugeno utility functions extending
those of polynomial functions given in Subsection~\ref{subsec:Pol} as well as
those of quasi-polynomial functions given in \cite{CouMar3}. For the sake of
simplicity, given $\varphi_{i}\colon L_{i}\rightarrow L$, $i\in\lbrack n]$, we
make use of the shorthand notation $\overline{\varphi}(\mathbf{x}%
)=(\varphi_{1}(x_{1}),\ldots,\varphi_{n}(x_{n}))$ and $\overline{\varphi
}^{\,-1}(c)=\{\mathbf{d}:\varphi_{i}(d_{i})=c$ for all $i\in\left[  n\right]
\}$, for every $c\in{\mathcal{R}}$.

We say that a function $f\colon\prod_{i\in\lbrack n]}L_{i}\rightarrow L$ is
\emph{pseudo-max homogeneous} (resp. \emph{pseudo-min homogeneous}) if there
are local utility functions $\varphi_{i}\colon L_{i}\rightarrow L$,
$i\in\lbrack n]$, such that for every $\mathbf{x}\in\prod_{i\in\lbrack
n]}L_{i}$ and $c\in{\mathcal{R}}$,
\begin{equation}
f(\mathbf{x}\vee\mathbf{d})=f(\mathbf{x})\vee c~(\mbox{resp. }\,f(\mathbf{x}%
\wedge\mathbf{d})=f(\mathbf{x})\wedge c),~\mbox{whenever }\mathbf{d}%
\in\overline{\varphi}^{\,-1}(c).
\end{equation}

\begin{fact}
\label{fact:1} Let $f\colon\prod_{i\in[n]}L_{i}\rightarrow L$ be a function,
and let $\varphi_{i}\colon L_{i}\to L$, $i\in[n]$, be local utility functions.
If $f$ is pseudo-min homogeneous and pseudo-max homogeneous w.r.t.
$\varphi_{1},\ldots,\varphi_{n}$, then it satisfies the condition
\begin{equation}
\label{PseudoIdemp}\mbox{for every $c\in {\mathcal{R}}$ and $\mathbf{d}\in  \overline{\varphi}^{\, -1}(c)$, $f(\mathbf{d}) =c$.}
\end{equation}

\end{fact}

\begin{lemma}
\label{lemma:2} If $f(x_{1},\ldots,x_{n})=q(\varphi(x_{1}),\ldots,\varphi
_{n}(x_{n}))$ for some Sugeno integral $q\colon L^{n} \to L$ and local utility
functions $\varphi_{1},\ldots,\varphi_{n}$, then $f$ is pseudo-min homogeneous
and pseudo-max homogeneous w.r.t. $\varphi_{1},\ldots,\varphi_{n}$.
\end{lemma}

\begin{proof}
Let ${\mathcal{R}}$ be the common range of $\varphi_{1},\ldots,\varphi_{n}$,
let $c\in{\mathcal{R}}$ and $\mathbf{d}\in\overline{\varphi}^{\, -1}(c)$. By
Theorem~\ref{mainChar} and the fact that each $\varphi_{k}$ is
order-preserving, we have
\begin{align*}
f(\mathbf{x}\vee\mathbf{d})  &  = q(\overline{\varphi}(\mathbf{x}%
\vee\mathbf{d})) =q(\overline{\varphi}(\mathbf{x})\vee\overline{\varphi
}(\mathbf{d}))\\
&  = q(\overline{\varphi}(\mathbf{x})\vee c) =q(\overline{\varphi}%
(\mathbf{x}))\vee c = f(\mathbf{x})\vee c,
\end{align*}
and hence, $f$ is pseudo-max homogeneous. The dual statement follows similarly.
\end{proof}

For $\mathbf{x}, \mathbf{d}\in\prod_{i\in[n]}L_{i}$, let $[\mathbf{x}%
]_{\mathbf{d}}$ be the $n$-tuple whose $i$th component is $0_{L_{i}}$, if
$x_{i}\leq d_{i}$, and $x_{i}$, otherwise, and dually let $[\mathbf{x}%
]^{\mathbf{d}}$ be the $n$-tuple whose $i$th component is $1_{L_{i}}$, if
$x_{i}\geq d_{i}$, and $x_{i}$, otherwise. We say that $f\colon\prod_{i\in
[n]}L_{i}\rightarrow L$ is \emph{pseudo-horizontally maxitive} (resp.
\emph{pseudo-horizontally minitive}) if there are local utility functions
$\varphi_{i}\colon L_{i}\to L$, $i\in[n]$, such that for every $\mathbf{x}%
\in\prod_{i\in[n]}L_{i}$ and $c\in{\mathcal{R}}$, if $\mathbf{d}\in
\overline{\varphi}^{\, -1}(c)$, then
\begin{equation}
\label{eq:Hor}f(\mathbf{x}) = f(\mathbf{x}\wedge\mathbf{d})\vee f([\mathbf{x}%
]_{\mathbf{d}}) \quad(\mbox{resp. } \, f(\mathbf{x}) = f(\mathbf{x}%
\vee\mathbf{d})\wedge f([\mathbf{x}]^{\mathbf{d}})).
\end{equation}

\begin{lemma}
\label{lemma:3} If $f\colon\prod_{i\in[n]}L_{i}\rightarrow L$ is
order-preserving, pseudo-horizontally minitive (resp.\ pseudo-horizontally
maxitive) and satisfies \textup{(\ref{PseudoIdemp})}, then it is pseudo-min
homogeneous (resp. pseudo-max homogeneous).
\end{lemma}

\begin{proof}
If $f\colon\prod_{i\in\lbrack n]}L_{i}\rightarrow L$ is order-preserving,
pseudo-horizontally minitive and satisfies (\ref{PseudoIdemp}) w.r.t.
$\varphi_{1},\ldots,\varphi_{n}$, then for every $\mathbf{x}\in\prod
_{i\in\lbrack n]}L_{i}$, $c\in{\mathcal{R}}$, $\mathbf{d}\in\overline{\varphi
}^{\,-1}(c)$ we have%
\begin{align*}
f(\mathbf{x})\wedge c  &  =f(\mathbf{x})\wedge f(\mathbf{d})~\geq
~f(\mathbf{x}\wedge\mathbf{d})~=~f((\mathbf{x}\wedge\mathbf{d})\vee
\mathbf{d})\wedge f([\mathbf{x}\wedge\mathbf{d}]^{\mathbf{d}})\\
&  =f(\mathbf{d})\wedge f([\mathbf{x}]^{\mathbf{d}})~\geq~f(\mathbf{d})\wedge
f(\mathbf{x})~=~f(\mathbf{x})\wedge c.
\end{align*}
Hence $f$ is pseudo-min homogeneous w.r.t. $\varphi_{1},\ldots,\varphi_{n}$.
The dual statement can be proved similarly.
\end{proof}

\begin{lemma}
\label{lemma:4} Suppose that $f\colon\prod_{i\in[n]}L_{i}\rightarrow L$ is
order-preserving and pseudo-min homogeneous (resp. pseudo-max homogeneous),
and satisfies \textup{(\ref{PseudoIdemp})}. Then $f$ is pseudo-max homogeneous
(resp. pseudo-min homogeneous) if and only if it is pseudo-horizontally
maxitive (resp.\ pseudo-horizontally minitive).
\end{lemma}

\begin{proof}
Suppose that $f\colon\prod_{i\in[n]}L_{i}\rightarrow L$ is order-preserving
and pseudo-min homogeneous and satisfies (\ref{PseudoIdemp}) w.r.t.
$\varphi_{1},\ldots,\varphi_{n}$. Assume first that $f$ is pseudo-max
homogeneous w.r.t. $\varphi_{1},\ldots,\varphi_{n}$. For every $\mathbf{x}%
\in\prod_{i\in[n]}L_{i}$ and $\mathbf{d}\in\overline{\varphi}^{\, -1}(c)$,
where $c\in{\mathcal{R}}$, we have
\begin{align*}
f(\mathbf{x}\wedge\mathbf{d})\vee f([\mathbf{x}]_{\mathbf{d}})  &  =
\big(f(\mathbf{x})\wedge c\big)\vee f([\mathbf{x}]_{\mathbf{d}}) =
\big(f(\mathbf{x})\vee f([\mathbf{x}]_{\mathbf{d}})\big)\wedge\big(c\vee
f([\mathbf{x}]_{\mathbf{d}})\big)\\
&  = f(\mathbf{x})\wedge f(\mathbf{d}\vee[\mathbf{x}]_{\mathbf{d}}) =
f(\mathbf{x}),
\end{align*}
and hence $f$ is pseudo-horizontally maxitive w.r.t. $\varphi_{1}%
,\ldots,\varphi_{n}$.

Conversely, if $f$ is pseudo-horizontally maxitive w.r.t. $\varphi_{1}%
,\ldots,\varphi_{n}$, then by Lemma~\ref{lemma:3} $f$ is pseudo-max
homogeneous w.r.t. $\varphi_{1},\ldots,\varphi_{n}$. The dual statement can be
proved similarly.
\end{proof}

\begin{lemma}
\label{claim} If $f\colon\prod_{i\in[n]}L_{i}\rightarrow L$ is
order-preserving, pseudo-min homogeneous and pseudo-horizontally maxitive,
then it is pseudo-median decomposable w.r.t. local utility functions.
\end{lemma}

\begin{proof}
Let $\mathbf{x}\in\prod_{i\in\lbrack n]}L_{i}$ and let $k\in\lbrack n]$. If
$f$ is pseudo-horizontally maxitive, say w.r.t. $\varphi_{1},\ldots
,\varphi_{n}$, then $f(\mathbf{x})=f(\mathbf{x}\wedge\mathbf{d})\vee
f([\mathbf{x}]_{\mathbf{d}})$ holds for every $\mathbf{d}\in\overline{\varphi
}^{\,-1}(\varphi_{k}(x_{k}))$ whose $k$th component is $x_{k}$. Now if $f$ is
pseudo-min homogeneous, then $f(\mathbf{x}\wedge\mathbf{d})=f(\mathbf{x}%
_{k}^{1}\wedge\mathbf{d})=f(\mathbf{x}_{k}^{1})\wedge\varphi_{k}(x_{k})$, and
by the definition of $[\mathbf{x}]_{\mathbf{d}}$, we have $f([\mathbf{x}%
]_{\mathbf{d}})\leq f(\mathbf{x}_{k}^{0})$. Thus,
\begin{align*}
f(\mathbf{x})  &  =\mathop{\rm med}\nolimits\big(f(\mathbf{x}_{k}%
^{0}),f(\mathbf{x}),f(\mathbf{x}_{k}^{1})\big)=\big(f(\mathbf{x}_{k}^{0})\vee
f(\mathbf{x})\big)\wedge f(\mathbf{x}_{k}^{1})\\
&  =\big(f(\mathbf{x}_{k}^{0})\vee(f(\mathbf{x}_{k}^{1})\wedge\varphi
_{k}(x_{k}))\big)\wedge f(\mathbf{x}_{k}^{1})=f(\mathbf{x}_{k}^{0}%
)\vee\big(f(\mathbf{x}_{k}^{1})\wedge\varphi_{k}(x_{k})\big)\\
&  =\mathop{\rm med}\nolimits\big(f(\mathbf{x}_{k}^{0}),\varphi_{k}%
(x_{k}),f(\mathbf{x}_{k}^{1})\big).
\end{align*}
Since this holds for every $\mathbf{x}\in\prod_{i\in\lbrack n]}L_{i}$ and
$k\in\lbrack n]$, $f$ is pseudo-median decomposable.
\end{proof}

We can also extend the comonotonic properties as follows. We say that a
function $f\colon\prod_{i\in\lbrack n]}L_{i}\rightarrow L$ is
\emph{pseudo-comonotonic minitive} (resp. \emph{pseudo-comonotonic maxitive})
if there are local utility functions $\varphi_{i}\colon L_{i}\rightarrow L$,
$i\in\lbrack n]$, such that for every $\mathbf{x}$ and $\mathbf{x}^{\prime}$,
if $\overline{\varphi}(\mathbf{x})$ and $\overline{\varphi}(\mathbf{x}%
^{\prime})$ are comonotonic, then%
\[
f(\mathbf{x}\wedge\mathbf{x}^{\prime})=f(\mathbf{x})\wedge f(\mathbf{x}%
^{\prime})\quad
\mbox{(resp. $ f(\mathbf{x}\vee \mathbf{x}') = f(\mathbf{x})\vee f(\mathbf{x}')).$}
\]

The following fact is straightforward.

\begin{fact}
\label{claim0} Every Sugeno utility function of the form
\textup{(\ref{eq:SugenoUtility})} is pseudo-comono\-tonic minitive and
maxitive. Moreover, if a function is pseudo-comonotonic minitive (resp.
pseudo-comonotonic maxitive) and satisfies \textup{(\ref{PseudoIdemp})}, then
it is pseudo-min homogeneous (resp. pseudo-max homogeneous).
\end{fact}

Let $\mathbf{P}$ be the set comprising the properties of pseudo-min
homogeneity, pseudo-horizontal minitivity and pseudo-comonotic minitivity, and
let $\mathbf{P}^{d}$ be the set comprising the corresponding dual properties.
The following result generalizes the various characterizations of polynomial
functions given in Subsection~\ref{subsec:Pol}.

\begin{theorem}
\label{thm:LastMain} Let $f\colon\prod_{i\in[n]}L_{i}\rightarrow L$ be an
order-preserving function.
The following assertions are equivalent:

\begin{enumerate}
\item[$(i)$] $f$ is a Sugeno utility function.

\item[$(ii)$] $f$ is pseudo-median decomposable w.r.t. local utility functions.

\item[$(iii)$] $f$ is $P_{1}\in\mathbf{P}$ and $P_{2}\in\mathbf{P}^{d}$, and
satisfies \textup{(\ref{PseudoIdemp})}.
\end{enumerate}


\end{theorem}

\begin{proof}
By Corollary~\ref{Cor:Med}, we have $(i)\Leftrightarrow(ii)$. By
Lemma~\ref{lemma:2}, we also have that if $(i)$ holds, then $f$ is pseudo-min
homogeneous and pseudo-max homogeneous. Furthermore, by Fact~\ref{claim0} and
Lemmas~\ref{lemma:3},~\ref{lemma:4} and \ref{claim}, we have that any two
formulations of $(iii)$ are equivalent. By Lemma~\ref{claim},
$(iii)\Rightarrow(ii)$.

\end{proof}

\begin{remark}
By Fact~\ref{fact:1}, if $P_{1}$ and $P_{2}$ are the pseudo-homogeneity
properties, then (\ref{PseudoIdemp}) becomes redundant in $(iii)$. Similarly,
by Lemma~\ref{claim}, Corollary~\ref{Cor:Med}, and $(i)\Rightarrow(iii)$ of
Theorem~\ref{thm:LastMain}, if $P_{1}$ is pseudo-min homogeneity
(pseudo-horizontal minitivity) property, and $P_{2}$ is pseudo-horizontal
maxitivity (pseudo-max homogeneity) property, then (\ref{PseudoIdemp}) is
redundant in $(iii)$.
\end{remark}

\begin{remark}
Note that if a function is pseudo-comonotonic minitive or pseudo-comonotonic
maxitive (w.r.t. $\varphi_{k}\colon L_{k}\rightarrow L$, $k\in\left[
n\right]  $), then it is order-preserving on every set
\[
S_{\varphi,\sigma}^{n}=\bigl\{\mathbf{x}\in\prod_{i\in\lbrack n]}%
L_{i}:\overline{\varphi}(\mathbf{x})\in L_{\sigma}^{n}\bigr\}\subseteq
\prod_{i\in\lbrack n]}L_{i}.
\]
As it turns out, this fact can be extended to the whole domain $\prod
_{i\in\lbrack n]}L_{i}.$ To illustrate, let $\mathbf{x}=(x_{1},x_{2}%
,\dots,x_{n})\in L^{n}$ and $\mathbf{y}=(y_{1},x_{2},\dots,x_{n})\in L^{n}$
such that $\overline{\varphi}(\mathbf{x})$ and $\overline{\varphi}%
(\mathbf{y})$ are not comonotonic, say%
\begin{align*}
\varphi_{1}\left(  x_{1}\right)   &  <\varphi_{2}\left(  x_{2}\right)
\leq\varphi_{3}\left(  x_{3}\right)  \leq\cdots\leq\varphi_{n}\left(
x_{n}\right)  ,\\
\varphi_{2}\left(  x_{2}\right)   &  <\varphi_{1}\left(  y_{1}\right)
\leq\varphi_{3}\left(  x_{3}\right)  \leq\cdots\leq\varphi_{n}\left(
x_{n}\right)  .
\end{align*}
Since $\varphi_{1}$ and $\varphi_{2}$ have the same range, there exists
$z_{1}\in L_{1}$ such that $\varphi_{1}(z_{1})=\varphi_{2}(x_{2})$. Then, for
$\mathbf{z}=(z_{1},x_{2},\dots,x_{n})$, $\overline{\varphi}(\mathbf{z})$ is
comonotonic with $\overline{\varphi}(\mathbf{x})$ and $\overline{\varphi
}(\mathbf{y})$, and $\mathbf{x}<\mathbf{z}<\mathbf{y}$. Now, if $f\colon
\prod_{i\in\lbrack n]}L_{i}\rightarrow L$ is pseudo-comonotonic minitive or
pseudo-comonotonic maxitive (w.r.t. $\varphi_{k}\colon L_{k}\rightarrow L$,
$k=1,\ldots,n$), then $f(\mathbf{x})\leq f(\mathbf{z})\leq f(\mathbf{y})$. The
same idea, taking middle-points and applying it componentwise, can be used to
show that if a function is pseudo-comonotonic minitive or pseudo-comonotonic
maxitive, then it is order-preserving.
\end{remark}

\section{Factorization of Sugeno utility functions\label{sect construction}}

In this section we present an algorithm that decides whether a given function
$f\colon\prod_{i\in\lbrack n]}L_{i}\rightarrow L$ has a factorization of the
form (\ref{eq:SugenoUtility}) and constructs such a factorization if one
exists. The algorithm terminates in a finite number of steps only if the
chains $L_{1},\ldots,L_{n}$ are finite, but the construction behind the
algorithm works for infinite bounded chains as well. Therefore we state the
main result of this section (Theorem~\ref{thm alg correct}) without the
finiteness assumption, allowing the algorithm to perform infinitely many steps
to produce the desired output. However, we will need to make the additional
assumption that the chain $L$ is \emph{complete}, i.e., that every subset
$S\subseteq L$ has an infimum (denoted by $\bigwedge S$) and a supremum
(denoted by $\bigvee S$). Clearly, every finite chain and every closed real
interval is complete.

To ensure that the algorithm works correctly, we will also need two reasonable
assumptions on the function $f$. The first is that $f$ has no inessential
variables, i.e., it depends on all of its variables. If this is not the case,
e.g., $f$ does not depend on its first variable, then there is a function
$g\colon L_{2}\times\cdots\times L_{n}\rightarrow L$ such that $f\left(
x_{1},\ldots,x_{n}\right)  =g\left(  x_{2},\ldots,x_{n}\right)  $. Thus we can
apply the algorithm to the function $g$ instead of $f$ (if $g$ still has
inessential variables, then we can eliminate them in a similar way). The
second assumption is that%
\begin{equation}
f\left(  \mathbf{0}\right)  =0\text{ and }f\left(  \mathbf{1}\right)
=1.\label{eq BC for f}%
\end{equation}
If this condition is not met, then the parts of $L$ that lie outside the
interval $\left[  f\left(  \mathbf{0}\right)  ,f\left(  \mathbf{1}\right)
\right]  $ are negligible; we may remove them without changing the problem.
However, we do not need the assumption that the local utility functions
$\varphi_{i}$ share the same range $\mathcal{R}$.

In Subsection~\ref{subsect algo} we first give the intuitive idea behind our
construction, and then present Algorithm~\ref{SUFF} (Sugeno Utility Function
Factorization or SUFF for short). We work out an example in
Subsection~\ref{subsect example}, and in Subsection~\ref{subsect correct} we
prove the correctness of algorithm SUFF.

\subsection{The algorithm SUFF\label{subsect algo}}

To present the basic idea of our algorithm, let us suppose that $f(\mathbf{x}%
)=q(\varphi_{1}(x_{1}),\ldots,\varphi_{n}(x_{n}))$ is a Sugeno utility
function, and let us try to extract information about the local utility
functions $\varphi_{k}$ from the overall utility function $f$. For notational
simplicity, we consider only the case $k=1$; the other cases can be treated
similarly. In this case, the pseudo-median decomposition formula
(\ref{medianDecomp}) takes the form%
\[
f\left(  x_{1},x_{2},\ldots,x_{n}\right)  =\operatorname{med}\left(  f\left(
0,x_{2},\ldots,x_{n}\right)  ,\varphi_{1}\left(  x_{1}\right)  ,f\left(
1,x_{2},\ldots,x_{n}\right)  \right)  .
\]
By fixing the variables $x_{2},\ldots,x_{n}$, the left hand side becomes a
unary function in the variable $x_{1}$, and on the right hand side we have the
median of $\varphi_{1}\left(  x_{1}\right)  $ and the two constants
$s=f\left(  0,x_{2},\ldots,x_{n}\right)  ,t=f\left(  1,x_{2},\ldots
,x_{n}\right)  $.%

\begin{figure}
[ptb]
\begin{center}
\includegraphics[
natheight=5.393500cm,
natwidth=5.393500cm,
height=5.3642cm,
width=5.3642cm
]%
{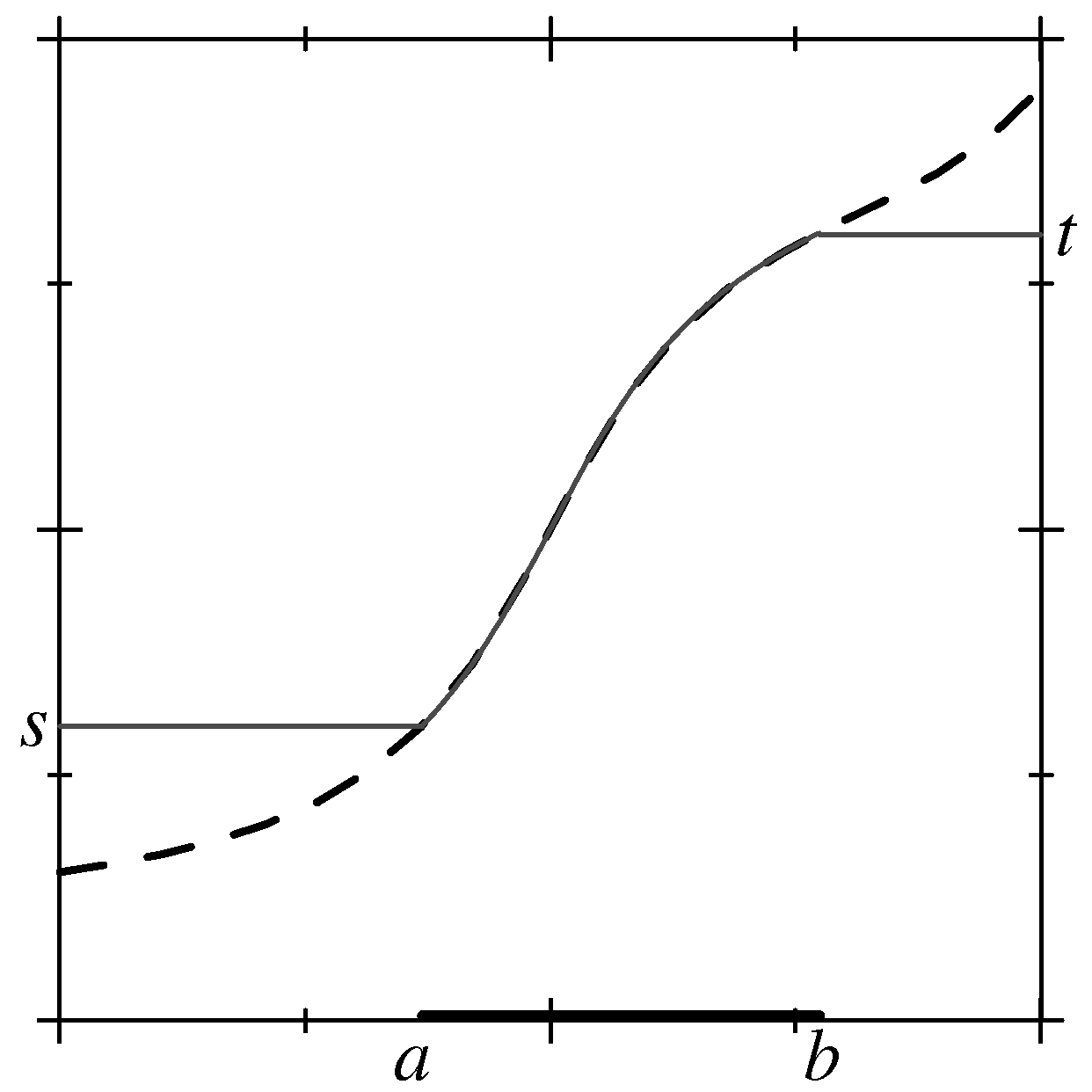}%
\caption{The graph of $\varphi_{1}$ as seen through a window}%
\label{fig2}%
\end{center}
\end{figure}

Figure~\ref{fig2} depicts this situation, where $L_{1}$ and $L$ are chosen to
be the unit interval $\left[  0,1\right]  \subseteq\mathbb{R}$, and the graphs
of $f\left(  x_{1},x_{2},\ldots,x_{n}\right)  $ and $\varphi_{1}\left(
x_{1}\right)  $ are represented by solid and dashed curves, respectively.
Observe that these two curves coincide on the interval $\left]  a,b\right[
=\left\{  x_{1}\in L_{1}\colon s<f\left(  x_{1},x_{2},\ldots,x_{n}\right)
<t\right\}  $, in other words, we can see some part of the graph of
$\varphi_{1}$ through the \textquotedblleft window\textquotedblright\ $\left]
a,b\right[  $. To the left of this window $s$ gives an upper bound for
$\varphi_{1}\left(  x_{1}\right)  $, while on the right hand side of the
window $t$ gives a lower bound. By fixing the variables $x_{2},\ldots,x_{n}$
to some other values, we may open other windows which may expose other parts
of the graph of $\varphi_{1}$. If we could find sufficiently many windows,
then we could recover $\varphi_{1}$. Unfortunately, this is not always the
case. In fact, as we shall see in the example of
Subsection~\ref{subsect example}, the local utility functions are not always
uniquely determined by $f$.

For any given $x_{1}\in L_{1}$, let us collect the tuples $\left(
x_{2},\ldots,x_{n}\right)  $ that open a window to $\varphi_{1}\left(
x_{1}\right)  $ into the set $\mathcal{W}_{x_{1}}$. Similarly, let
$\mathcal{L}_{x_{1}}$ and $\mathcal{U}_{x_{1}}$ be the sets of tuples that
provide only lower and upper bounds, respectively, and let $\mathcal{E}%
_{x_{1}}$ contain the remaining tuples of $L_{2}\times\cdots\times L_{n}$:%
\begin{align*}
\mathcal{W}_{x_{1}}  &  =\left\{  \left(  x_{2},\ldots,x_{n}\right)  :f\left(
0,x_{2},\ldots,x_{n}\right)  <f\left(  x_{1},x_{2},\ldots,x_{n}\right)
<f\left(  1,x_{2},\ldots,x_{n}\right)  \right\}  ,\\
\mathcal{L}_{x_{1}}  &  =\left\{  \left(  x_{2},\ldots,x_{n}\right)  :f\left(
0,x_{2},\ldots,x_{n}\right)  <f\left(  x_{1},x_{2},\ldots,x_{n}\right)
=f\left(  1,x_{2},\ldots,x_{n}\right)  \right\}  ,\\
\mathcal{U}_{x_{1}}  &  =\left\{  \left(  x_{2},\ldots,x_{n}\right)  :f\left(
0,x_{2},\ldots,x_{n}\right)  =f\left(  x_{1},x_{2},\ldots,x_{n}\right)
<f\left(  1,x_{2},\ldots,x_{n}\right)  \right\}  ,\\
\mathcal{E}_{x_{1}}  &  =\left\{  \left(  x_{2},\ldots,x_{n}\right)  :f\left(
0,x_{2},\ldots,x_{n}\right)  =f\left(  x_{1},x_{2},\ldots,x_{n}\right)
=f\left(  1,x_{2},\ldots,x_{n}\right)  \right\}  .
\end{align*}
Observe that $\mathcal{E}_{x_{1}}$ bears no information on $x_{1}$; we only
introduce it for notational convenience. Furthermore, let us define the sets
$\mathcal{W}_{x_{1}}^{f},\mathcal{L}_{x_{1}}^{f},\mathcal{U}_{x_{1}}^{f}$ as
follows:%
\begin{align*}
\mathcal{W}_{x_{1}}^{f}  &  =\left\{  f\left(  x_{1},x_{2},\ldots
,x_{n}\right)  \colon\left(  x_{2},\ldots,x_{n}\right)  \in\mathcal{W}_{x_{1}%
}\right\}  ,\\
\mathcal{L}_{x_{1}}^{f}  &  =\left\{  f\left(  x_{1},x_{2},\ldots
,x_{n}\right)  \colon\left(  x_{2},\ldots,x_{n}\right)  \in\mathcal{L}_{x_{1}%
}\right\}  ,\\
\mathcal{U}_{x_{1}}^{f}  &  =\left\{  f\left(  x_{1},x_{2},\ldots
,x_{n}\right)  \colon\left(  x_{2},\ldots,x_{n}\right)  \in\mathcal{U}_{x_{1}%
}\right\}  .
\end{align*}

Note that $\mathcal{W}_{x_{1}}^{f}$ cannot have more than one element, for
otherwise $f$ is not a Sugeno utility function. If $\mathcal{W}_{x_{1}}^{f}$
is a one-element set, then let $w_{x_{1}}$ denote its unique element:%
\begin{equation}
w_{x_{1}}=f\left(  x_{1},x_{2},\ldots,x_{n}\right)  \text{ if }\left(
x_{2},\ldots,x_{n}\right)  \in\mathcal{W}_{x_{1}}. \label{eq w}%
\end{equation}
Furthermore, let $l_{x_{1}}$ and $u_{x_{1}}$ be given as follows:%
\begin{align}
l_{x_{1}}  &  =\bigvee\mathcal{L}_{x_{1}}^{f}\text{ if }\mathcal{L}_{x_{1}%
}\neq\emptyset,\label{eq l}\\
u_{x_{1}}  &  =\bigwedge\mathcal{U}_{x_{1}}^{f}\text{ if }\mathcal{U}_{x_{1}%
}\neq\emptyset. \label{eq u}%
\end{align}
If any of the sets $\mathcal{W}_{x_{1}},\mathcal{L}_{x_{1}},\mathcal{U}%
_{x_{1}}$ is empty, then the corresponding values $w_{x_{1}},l_{x_{1}%
},u_{x_{1}}$ are undefined. From the above considerations it is clear that
$\varphi_{1}$ satisfies the (in)equalities%
\begin{equation}
\varphi_{1}\left(  x_{1}\right)  =w_{x_{1}},~\varphi_{1}\left(  x_{1}\right)
\geq l_{x_{1}},~\varphi_{1}\left(  x_{1}\right)  \leq u_{x_{1}},
\label{eq consistency fi}%
\end{equation}
whenever the right hand sides are defined.

Let us define a function $\varphi_{1}^{f}\colon L_{1}\rightarrow L$ by making
use of the following four rules:

\begin{enumerate}
\item[(W)] if $\mathcal{W}_{x_{1}}\neq\emptyset$ then let $\varphi_{1}%
^{f}\left(  x_{1}\right)  =w_{x_{1}}$;

\item[(L)] if $\mathcal{W}_{x_{1}}=\emptyset,\mathcal{L}_{x_{1}}\neq
\emptyset,\mathcal{U}_{x_{1}}=\emptyset$ then let $\varphi_{1}^{f}\left(
x_{1}\right)  =l_{x_{1}}$;

\item[(U)] if $\mathcal{W}_{x_{1}}=\emptyset,\mathcal{L}_{x_{1}}%
=\emptyset,\mathcal{U}_{x_{1}}\neq\emptyset$ then let $\varphi_{1}^{f}\left(
x_{1}\right)  =u_{x_{1}}$;

\item[(LU)] if $\mathcal{W}_{x_{1}}=\emptyset,\mathcal{L}_{x_{1}}\neq
\emptyset,\mathcal{U}_{x_{1}}\neq\emptyset$ then let $\varphi_{1}^{f}\left(
x_{1}\right)  =l_{x_{1}}$.
\end{enumerate}

\noindent Observe that the four cases above cover all possibilities since
$\mathcal{W}_{x_{1}}=\mathcal{U}_{x_{1}}=\mathcal{L}_{x_{1}}=\emptyset$ is
ruled out by the assumption that $f$ depends on its first variable. It is
important to note that $\varphi_{1}^{f}$ is computed only from $f$, without
reference to $\varphi_{1}$.

We can define functions $\varphi_{k}^{f}\colon L_{k}\rightarrow L$ for each
$k\in\left[  n\right]  $ in a similar manner, and we will prove that if $f$ is
a Sugeno utility function, then these are local utility functions and they
provide a factorization $f\left(  \mathbf{x}\right)  =q^{f}\bigl(\varphi
_{1}^{f}\left(  x_{1}\right)  ,\ldots,\varphi_{n}^{f}\left(  x_{n}\right)
\bigr)$, where $q^{f}$ is the Sugeno integral given in
Theorem~\ref{thm dnf for pseudo-Sugeno}:%
\[
q^{f}\left(  y_{1},\ldots,y_{n}\right)  =\bigvee\limits_{I\subseteq\left[
n\right]  }\bigl(f\left(  \mathbf{e}_{I}\right)  \wedge\bigwedge\limits_{i\in
I}y_{i}\bigr).
\]
Note that (\ref{eq BC for f}) implies that the polynomial $q^{f}$ is indeed a
Sugeno integral.

Algorithm~\ref{SUFF}, which will be referred to as algorithm SUFF in the
sequel, summarizes the construction of the local utility functions
$\varphi_{k}^{f}$ and the Sugeno integral $q^{f}$. The value \textbf{false} is
returned if

\begin{itemize}
\item $f$ is not order-preserving (line~\ref{algline OP not OK}),

\item several different values for $\varphi_{k}\left(  x_{k}\right)  $ are
seen through some windows (line \ref{algline W not OK}),

\item the values $l_{x_{1}},w_{x_{1}},u_{x_{1}}$ are contradictory
(line~\ref{algline l<w<u not OK}), or

\item $f$ does not depend on all of its variables
(line~\ref{algline ess not OK}).
\end{itemize}

\noindent Otherwise the output is $q^{f}$ and $\varphi_{k}^{f}\left(
k\in\left[  n\right]  \right)  $, which are computed as explained above.

In the next subsection we will prove the following theorem, which ensures the
correctness of algorithm SUFF.

\begin{algorithm}
\caption{Sugeno Utility Function Factorization}
\label{SUFF}
\begin{algorithmic}[1]
\renewcommand{\algorithmiccomment}[1]{\hfill // #1}
\REQUIRE $f$ depends on all of its variables and satisfies  (\ref{eq BC for f})
\IF{$f$ is not order-preserving}
\RETURN {\bf false} \COMMENT{$f$ is not a SUF}\label{algline OP not OK}
\ENDIF
\FOR{$k\in \left[ n\right] $}
\FOR{$x_{k}\in L_{k}$}
\STATE compute $\mathcal{W}^{f}_{x_{k}}$
\IF{$\left\vert \mathcal{W}^{f}_{x_{k}} \right\vert \geq 2$}
\RETURN {\bf false} \COMMENT{$f$ is not a SUF}\label{algline W not OK}
\ENDIF
\STATE compute $\mathcal{L}^{f}_{x_{k}}, \mathcal{U}^{f}_{x_{k}}$ and $w_{x_{k}}, l_{x_{k}}, u_{x_{k}}$
\IF{$ l_{x_{k}}>u_{x_{k}}$ {\bf or} $ l_{x_{k}}>w_{x_{k}}$ {\bf or} $w_{x_{k}}>u_{x_{k}}$}
\RETURN {\bf false} \COMMENT{$f$ is not a SUF}\label{algline l<w<u not OK}
\ENDIF
\IF{$\mathcal{W}_{x_{k}}\neq \emptyset $}
\STATE $\varphi _{k}^{f}\left(x_{k}\right) :=w_{x_{k}}$ \COMMENT{(W)}
\ELSIF{$\mathcal{L}_{x_{k}}\neq \emptyset$}
\STATE $\varphi _{k}^{f}\left( x_{k}\right):=l_{x_{k}}$ \COMMENT{(L),(LU)}
\ELSIF{$\mathcal{U}_{x_{k}}\neq \emptyset $}
\STATE $\varphi _{k}^{f}\left( x_{k}\right):=u_{x_{k}}$ \COMMENT{(U)}
\ELSE \RETURN {\bf false} \COMMENT{$x_k$ is inessential}\label{algline ess not OK}
\ENDIF
\ENDFOR
\ENDFOR
\STATE compute $q^f$
\RETURN $q^{f},\varphi _{1}^{f},\ldots ,\varphi _{n}^{f}$ \COMMENT{$f$ is a SUF}
\end{algorithmic}
\end{algorithm}

\begin{theorem}
\label{thm alg correct}If $f\colon\prod_{i\in\lbrack n]}L_{i}\rightarrow L$ is
an order-preserving pseudo-Sugeno integral, then algorithm SUFF constructs a
Sugeno integral $q^{f}$ and local utility functions $\varphi_{1}^{f}%
,\ldots,\varphi_{n}^{f}$ such that
\[
f\left(  \mathbf{x}\right)  =q^{f}\bigl(\varphi_{1}^{f}\left(  x_{1}\right)
,\ldots,\varphi_{n}^{f}\left(  x_{n}\right)  \bigr).
\]
Otherwise, the algorithm outputs the value {\emph{\textbf{false}}}.
\end{theorem}

It is clear that every Sugeno utility function is an order-preserving
pseudo-Sugeno integral. Conversely, if $f$ is an order-preserving
pseudo-Sugeno integral, then the algorithm SUFF produces a factorization of
$f$ into a composition of a Sugeno integral and local utility functions by
Theorem~\ref{thm alg correct}. Thus $f$ is a Sugeno utility function.

\begin{corollary}
\label{cor SUF=OP PSI}The class of Sugeno utility functions coincides with the
class of order-preserving pseudo-Sugeno integrals.
\end{corollary}

Note that the same Sugeno utility function can have several different
factorizations, hence starting with a function $f\left(  \mathbf{x}\right)
=q\left(  \varphi_{1}\left(  x_{1}\right)  ,\ldots,\varphi_{n}\left(
x_{n}\right)  \right)  $, just as we did at the beginning of this subsection,
the factorization $f\left(  \mathbf{x}\right)  =q^{f}\bigl(\varphi_{1}%
^{f}\left(  x_{1}\right)  ,\ldots,\varphi_{n}^{f}\left(  x_{n}\right)  \bigr)$
provided by the algorithm may be a different one (see the example in the next subsection).

\subsection{An example\label{subsect example}}

Let us illustrate our construction with a concrete (albeit fictitious)
example. Customers evaluate hotels along three criteria, namely quality of
services, price, and whether the hotel has a good location. Service is
evaluated on a four-level scale $L_{1}$: \texttt{*}$<$\texttt{**}%
$<$\texttt{***}$<$\texttt{****}, price is evaluated on a three-level scale
$L_{2}$: \texttt{-}$<$\texttt{0}$<$\texttt{+} (where \textquotedblleft%
\texttt{-}\textquotedblright means expensive, thus less desirable, and
\textquotedblleft\texttt{+}\textquotedblright means cheap, thus more
desirable), and the third scale is $L_{3}$: \texttt{n(o)}$<$\texttt{y(es)}. In
addition, each hotel receives an overall rating on the scale $L:1<\cdots<8$,
which gives the overall utility function $f\colon L_{1}\times L_{2}\times
L_{3}\rightarrow L$ (see Table~\ref{table example OUF}). We will find a
factorization of this function, and we will analyze its structure in order to
draw conclusions about the nature of the \textquotedblleft human
aggregation\textquotedblright\ that the customers (unconsciously) perform when
forming their opinions about hotels.

\begin{table}[p]
\centering\renewcommand{\arraystretch}{1.2}
\begin{tabular}
[t]{cccc}\hline
\multicolumn{1}{|c}{\ service$\ $} & \multicolumn{1}{|c}{\ price$\ $} &
\multicolumn{1}{|c}{\ location$\ $} & \multicolumn{1}{||c|}{$\ f\ $%
}\\\hline\hline
\multicolumn{1}{|c}{\texttt{*}} & \multicolumn{1}{|c}{\texttt{-}} &
\multicolumn{1}{|c}{\texttt{n}} & \multicolumn{1}{||c|}{1}\\\hline
\multicolumn{1}{|c}{\texttt{**}} & \multicolumn{1}{|c}{\texttt{-}} &
\multicolumn{1}{|c}{\texttt{n}} & \multicolumn{1}{||c|}{2}\\\hline
\multicolumn{1}{|c}{\texttt{***}} & \multicolumn{1}{|c}{\texttt{-}} &
\multicolumn{1}{|c}{\texttt{n}} & \multicolumn{1}{||c|}{2}\\\hline
\multicolumn{1}{|c}{\texttt{****}} & \multicolumn{1}{|c}{\texttt{-}} &
\multicolumn{1}{|c}{\texttt{n}} & \multicolumn{1}{||c|}{2}\\\hline
\multicolumn{1}{|c}{\texttt{*}} & \multicolumn{1}{|c}{\texttt{0}} &
\multicolumn{1}{|c}{\texttt{n}} & \multicolumn{1}{||c|}{2}\\\hline
\multicolumn{1}{|c}{\texttt{**}} & \multicolumn{1}{|c}{\texttt{0}} &
\multicolumn{1}{|c}{\texttt{n}} & \multicolumn{1}{||c|}{2}\\\hline
\multicolumn{1}{|c}{\texttt{***}} & \multicolumn{1}{|c}{\texttt{0}} &
\multicolumn{1}{|c}{\texttt{n}} & \multicolumn{1}{||c|}{2}\\\hline
\multicolumn{1}{|c}{\texttt{****}} & \multicolumn{1}{|c}{\texttt{0}} &
\multicolumn{1}{|c}{\texttt{n}} & \multicolumn{1}{||c|}{2}\\\hline
\multicolumn{1}{|c}{\texttt{*}} & \multicolumn{1}{|c}{\texttt{+}} &
\multicolumn{1}{|c}{\texttt{n}} & \multicolumn{1}{||c|}{2}\\\hline
\multicolumn{1}{|c}{\texttt{**}} & \multicolumn{1}{|c}{\texttt{+}} &
\multicolumn{1}{|c}{\texttt{n}} & \multicolumn{1}{||c|}{2}\\\hline
\multicolumn{1}{|c}{\texttt{***}} & \multicolumn{1}{|c}{\texttt{+}} &
\multicolumn{1}{|c}{\texttt{n}} & \multicolumn{1}{||c|}{2}\\\hline
\multicolumn{1}{|c}{\texttt{****}} & \multicolumn{1}{|c}{\texttt{+}} &
\multicolumn{1}{|c}{\texttt{n}} & \multicolumn{1}{||c|}{2}\\\hline
\multicolumn{1}{|c}{\texttt{*}} & \multicolumn{1}{|c}{\texttt{-}} &
\multicolumn{1}{|c}{\texttt{y}} & \multicolumn{1}{||c|}{3}\\\hline
\multicolumn{1}{|c}{\texttt{**}} & \multicolumn{1}{|c}{\texttt{-}} &
\multicolumn{1}{|c}{\texttt{y}} & \multicolumn{1}{||c|}{3}\\\hline
\multicolumn{1}{|c}{\texttt{***}} & \multicolumn{1}{|c}{\texttt{-}} &
\multicolumn{1}{|c}{\texttt{y}} & \multicolumn{1}{||c|}{7}\\\hline
\multicolumn{1}{|c}{\texttt{****}} & \multicolumn{1}{|c}{\texttt{-}} &
\multicolumn{1}{|c}{\texttt{y}} & \multicolumn{1}{||c|}{8}\\\hline
\multicolumn{1}{|c}{\texttt{*}} & \multicolumn{1}{|c}{\texttt{0}} &
\multicolumn{1}{|c}{\texttt{y}} & \multicolumn{1}{||c|}{5}\\\hline
\multicolumn{1}{|c}{\texttt{**}} & \multicolumn{1}{|c}{\texttt{0}} &
\multicolumn{1}{|c}{\texttt{y}} & \multicolumn{1}{||c|}{5}\\\hline
\multicolumn{1}{|c}{\texttt{***}} & \multicolumn{1}{|c}{\texttt{0}} &
\multicolumn{1}{|c}{\texttt{y}} & \multicolumn{1}{||c|}{7}\\\hline
\multicolumn{1}{|c}{\texttt{****}} & \multicolumn{1}{|c}{\texttt{0}} &
\multicolumn{1}{|c}{\texttt{y}} & \multicolumn{1}{||c|}{8}\\\hline
\multicolumn{1}{|c}{\texttt{*}} & \multicolumn{1}{|c}{\texttt{+}} &
\multicolumn{1}{|c}{\texttt{y}} & \multicolumn{1}{||c|}{6}\\\hline
\multicolumn{1}{|c}{\texttt{**}} & \multicolumn{1}{|c}{\texttt{+}} &
\multicolumn{1}{|c}{\texttt{y}} & \multicolumn{1}{||c|}{6}\\\hline
\multicolumn{1}{|c}{\texttt{***}} & \multicolumn{1}{|c}{\texttt{+}} &
\multicolumn{1}{|c}{\texttt{y}} & \multicolumn{1}{||c|}{7}\\\hline
\multicolumn{1}{|c}{\texttt{****}} & \multicolumn{1}{|c}{\texttt{+}} &
\multicolumn{1}{|c}{\texttt{y}} & \multicolumn{1}{||c|}{8}\\\hline
&  &  &
\end{tabular}
\caption{Hotel example: the overall utility function}%
\label{table example OUF}%
\end{table}

First we apply Theorem~\ref{thm dnf for pseudo-Sugeno} to find the underlying
Sugeno integral:%
\begin{multline*}
q^{f}\left(  y_{1},y_{2},y_{3}\right)  =1\vee\left(  2\wedge y_{1}\right)
\vee\left(  2\wedge y_{2}\right)  \vee\left(  3\wedge y_{3}\right) \\
\vee\left(  2\wedge y_{1}\wedge y_{2}\right)  \vee\left(  8\wedge y_{1}\wedge
y_{3}\right)  \vee\left(  6\wedge y_{2}\wedge y_{3}\right)  \vee\left(
8\wedge y_{1}\wedge y_{2}\wedge y_{3}\right)  .
\end{multline*}
Since $1$ (resp. $8$) is the least (resp. greatest) element of $L$, this
polynomial function $q^{f}$ is indeed a Sugeno integral. We can simplify
$q^{f}$ by cancelling those terms which are absorbed by some other terms in
the disjunction:%
\[
q^{f}\left(  y_{1},y_{2},y_{3}\right)  =\left(  2\wedge y_{1}\right)
\vee\left(  2\wedge y_{2}\right)  \vee\left(  3\wedge y_{3}\right)
\vee\left(  y_{1}\wedge y_{3}\right)  \vee\left(  6\wedge y_{2}\wedge
y_{3}\right)  .
\]

We will be able to perform further simplifications after constructing the
local utility functions. Table~\ref{table example partitions} shows the
partitions of $L_{2}\times L_{3}=\mathcal{W}_{x_{1}}\cup\mathcal{L}_{x_{1}%
}\cup\mathcal{U}_{x_{1}}\cup\mathcal{E}_{x_{1}}$ corresponding to the four
possible elements $x_{1}\in L_{1}$. The numbers in parentheses are the values
of $f\left(  x_{1},x_{2},x_{3}\right)  $ (recall that we do not compute any
values for the sets $\mathcal{E}_{x_{1}}$); these are used to compute the
numbers $l_{x_{1}},w_{x_{1}},u_{x_{1}}$ shown in Table~\ref{table example LUF}%
. This table contains these data for all $x_{2}\in L_{2}$ and $x_{3}\in L_{3}$
as well, together with the values of $\varphi_{1}^{f}\left(  x_{1}\right)
,\varphi_{2}^{f}\left(  x_{2}\right)  ,\varphi_{3}^{f}\left(  x_{3}\right)  $.

\begin{table}[p]
\centering\renewcommand{\arraystretch}{1.2}
\begin{tabular}
[t]{ccccc}\hline
\multicolumn{1}{|c}{} & \multicolumn{1}{||c|}{\texttt{*}} &
\multicolumn{1}{|c}{\texttt{**}} & \multicolumn{1}{|c}{\texttt{***}} &
\multicolumn{1}{|c|}{\texttt{****}}\\\hline\hline
\multicolumn{1}{|c}{\texttt{\thinspace(-,n)\thinspace}} &
\multicolumn{1}{||l|}{$\ \mathcal{U}_{\text{\texttt{*}}}\left(  1\right)  \ $}
& \multicolumn{1}{|l|}{$\ \mathcal{L}_{\text{\texttt{**}}}\left(  2\right)
\ $} & \multicolumn{1}{|l|}{$\ \mathcal{L}_{\text{\texttt{***}}}\left(
2\right)  \ $} & \multicolumn{1}{|l|}{$\ \mathcal{L}_{\text{\texttt{****}}%
}\left(  2\right)  \ $}\\\hline
\multicolumn{1}{|c}{\texttt{\thinspace(0,n)\thinspace}} &
\multicolumn{1}{||l|}{$\ \mathcal{E}_{\text{\texttt{*}}}\ $} &
\multicolumn{1}{|l|}{$\ \mathcal{E}_{\text{\texttt{**}}}\ $} &
\multicolumn{1}{|l|}{$\ \mathcal{E}_{\text{\texttt{***}}}\ $} &
\multicolumn{1}{|l|}{$\ \mathcal{E}_{\text{\texttt{****}}}\ $}\\\hline
\multicolumn{1}{|c}{\texttt{\thinspace(+,n)\thinspace}} &
\multicolumn{1}{||l|}{$\ \mathcal{E}_{\text{\texttt{*}}}\ $} &
\multicolumn{1}{|l|}{$\ \mathcal{E}_{\text{\texttt{**}}}\ $} &
\multicolumn{1}{|l|}{$\ \mathcal{E}_{\text{\texttt{***}}}\ $} &
\multicolumn{1}{|l|}{$\ \mathcal{E}_{\text{\texttt{****}}}\ $}\\\hline
\multicolumn{1}{|c}{\texttt{\thinspace(-,y)\thinspace}} &
\multicolumn{1}{||l|}{$\ \mathcal{U}_{\text{\texttt{*}}}\left(  3\right)  \ $}
& \multicolumn{1}{|l|}{$\ \mathcal{U}_{\text{\texttt{**}}}\left(  3\right)
\ $} & \multicolumn{1}{|l|}{$\ \mathcal{W}_{\text{\texttt{***}}}\left(
7\right)  \ $} & \multicolumn{1}{|l|}{$\ \mathcal{L}_{\text{\texttt{****}}%
}\left(  8\right)  \ $}\\\hline
\multicolumn{1}{|c}{\texttt{\thinspace(0,y)\thinspace}} &
\multicolumn{1}{||l|}{$\ \mathcal{U}_{\text{\texttt{*}}}\left(  5\right)  \ $}
& \multicolumn{1}{|l|}{$\ \mathcal{U}_{\text{\texttt{**}}}\left(  5\right)
\ $} & \multicolumn{1}{|l|}{$\ \mathcal{W}_{\text{\texttt{***}}}\left(
7\right)  \ $} & \multicolumn{1}{|l|}{$\ \mathcal{L}_{\text{\texttt{****}}%
}\left(  8\right)  \ $}\\\hline
\multicolumn{1}{|c}{\texttt{\thinspace(+,y)\thinspace}} &
\multicolumn{1}{||l|}{$\ \mathcal{U}_{\text{\texttt{*}}}\left(  6\right)  \ $}
& \multicolumn{1}{|l|}{$\ \mathcal{U}_{\text{\texttt{**}}}\left(  6\right)
\ $} & \multicolumn{1}{|l|}{$\ \mathcal{W}_{\text{\texttt{***}}}\left(
7\right)  \ $} & \multicolumn{1}{|l|}{$\ \mathcal{L}_{\text{\texttt{****}}%
}\left(  8\right)  \ $}\\\hline
&  &  &  &
\end{tabular}
\caption{Hotel example: the partitions of $L_{2}\times L_{3}$}%
\label{table example partitions}%
\end{table}

\begin{table}[p]
\centering\renewcommand{\arraystretch}{1.2}
\begin{tabular}
[c]{c}%
\begin{tabular}
[t]{|p{0.75cm}||c|c|c|c|}\hline
& $\ l\ $ & $\ w\ $ & $\ u\ $ & $\ \varphi_{1}^{f}\ $\\\hline\hline
\texttt{\hspace*{\fill}*\hspace*{\fill}} &  &  & $1$ & $1$\\\hline
\texttt{\hspace*{\fill}**\hspace*{\fill}} & $2$ &  & $3$ & $2$\\\hline
\texttt{\hspace*{\fill}***\hspace*{\fill}} & $2$ & $7$ &  & $7$\\\hline
\texttt{\hspace*{\fill}****\hspace*{\fill}} & $8$ &  &  & $8$\\\hline
\end{tabular}
\\%
\begin{tabular}
[t]{|p{0.75cm}||c|c|c|c|}\hline
& $\ l\ $ & $\ w\ $ & $\ u\ $ & $\ \varphi_{2}^{f}\ $\\\hline\hline
\texttt{\hspace*{\fill}-\hspace*{\fill}} &  &  & $1$ & $1$\\\hline
\texttt{\hspace*{\fill}0\hspace*{\fill}} & $2$ & $5$ &  & $5$\\\hline
\texttt{\hspace*{\fill}+\hspace*{\fill}} & $6$ &  &  & $6$\\\hline
\end{tabular}
\\%
\begin{tabular}
[t]{p{0.75cm}cccc}\hline
\multicolumn{1}{|p{0.75cm}}{} & \multicolumn{1}{||c}{$\ l\ $} &
\multicolumn{1}{|c}{$\ w\ $} & \multicolumn{1}{|c}{$\ u\ $} &
\multicolumn{1}{|c|}{$\ \varphi_{3}^{f}\ $}\\\hline\hline
\multicolumn{1}{|p{0.75cm}}{\texttt{\hspace*{\fill}n\hspace*{\fill}}} &
\multicolumn{1}{||c}{} & \multicolumn{1}{|c}{} & \multicolumn{1}{|c}{$1$} &
\multicolumn{1}{|c|}{$1$}\\\hline
\multicolumn{1}{|p{0.75cm}}{\texttt{\hspace*{\fill}y\hspace*{\fill}}} &
\multicolumn{1}{||c}{$8$} & \multicolumn{1}{|c}{} & \multicolumn{1}{|c}{} &
\multicolumn{1}{|c|}{$8$}\\\hline
\multicolumn{1}{c}{} &  &  &  &
\end{tabular}
\end{tabular}
\caption{Hotel example: the local utility functions}%
\label{table example LUF}%
\end{table}

Now that we know that the greatest value of $\varphi_{2}^{f}$ is $6$, we can
simplify the Sugeno integral $q^{f}$ by replacing $6\wedge y_{2}\wedge y_{3}$
with $y_{2}\wedge y_{3}$, and \textquotedblleft factoring
out\textquotedblright\ $y_{1}\vee y_{2}$:%
\[
\left(  3\wedge y_{3}\right)  \vee\left(  \left(  y_{1}\vee y_{2}\right)
\wedge\left(  2\vee y_{3}\right)  \right)  =\operatorname{med}\left(  3\wedge
y_{3},y_{1}\vee y_{2},2\vee y_{3}\right)  .
\]
Note that this polynomial function is different from $q^{f}$, but it gives the
same overall utility function $f$. This example shows that the Sugeno integral
is not uniquely determined by $f$, and neither are the local utility functions
(e.g., we could have chosen $\varphi_{1}^{f}\left(  \text{\texttt{**}}\right)
=3$ according to Remark~\ref{rem fi order preserving}).

To better understand the behavior of $f$, let us separate two cases upon the
location of the hotel:%
\begin{align}
f\left(  x_{1},x_{2},x_{3}\right)   &  =\operatorname{med}\bigl(3\wedge
\varphi_{3}^{f}\left(  x_{3}\right)  ,\varphi_{1}^{f}\left(  x_{1}\right)
\vee\varphi_{2}^{f}\left(  x_{2}\right)  ,2\vee\varphi_{3}^{f}\left(
x_{3}\right)  \bigr)\label{eq analysis}\\
&  =\left\{
\begin{array}
[c]{rl}%
\bigl(\varphi_{1}^{f}\left(  x_{1}\right)  \vee\varphi_{2}^{f}\left(
x_{2}\right)  \bigr)\vee3,~ & \text{if }x_{3}=\text{\texttt{y}},\\
\bigl(\varphi_{1}^{f}\left(  x_{1}\right)  \vee\varphi_{2}^{f}\left(
x_{2}\right)  \bigr)\wedge2,~ & \text{if }x_{3}=\text{\texttt{n}}.
\end{array}
\right.  \nonumber
\end{align}
We can see from (\ref{eq analysis}) that once $x_{3}$ is fixed, what matters
is the higher one of $\varphi_{1}^{f}\left(  x_{1}\right)  $ and $\varphi
_{2}^{f}\left(  x_{2}\right)  $. Thus, instead of aiming at an average level
in both, a better strategy would be to maximize one of them. Moreover,
$\varphi_{1}^{f}$ either outputs very low or very high scores, whereas
$\varphi_{2}^{f}$ is almost maximized once the price is not very bad. Hence it
seems more reasonable to focus on service rather than on price. The third
variable can radically change the final outcome, but little can be done to
improve the location of the hotel.

\subsection{Proof of correctness\label{subsect correct}}

We assume that $L_{1},\ldots,L_{n},L$ are bounded chains, $L$ is complete,
$f\colon\prod_{i\in\lbrack n]}L_{i}\rightarrow L$ depends on all of its
variables and satisfies (\ref{eq BC for f}). If the output of algorithm SUFF
is not \textbf{false}, then it computes a Sugeno integral $q^{f}$ and
functions $\varphi_{k}^{f}\colon L_{k}\rightarrow L$ for each $k\in\left[
n\right]  $. It is clear from the construction that%
\begin{equation}
\varphi_{k}^{f}\left(  x_{k}\right)  =w_{x_{k}},~\varphi_{k}^{f}\left(
x_{k}\right)  \geq l_{x_{k}},~\varphi_{k}^{f}\left(  x_{k}\right)  \leq
u_{x_{k}} \label{eq consistency fif}%
\end{equation}
holds for all $k\in\left[  n\right]  ,x_{k}\in L_{k}$ (whenever the values on
the right hand sides are defined). To prove Theorem~\ref{thm alg correct} we
shall make use of two auxiliary lemmas. The first states that the functions
$\varphi_{k}^{f}$ are local utility functions, i.e., order-preserving functions.

\begin{lemma}
\label{lemma fi order preserving}If algorithm SUFF does not return the value
{\emph{\textbf{false}}}, then the functions $\varphi_{1}^{f},\ldots
,\varphi_{n}^{f}$ constructed by the algorithm are local utility functions.
\end{lemma}

\begin{proof}
We show that $\varphi_{1}^{f}$ is order-preserving; the other cases can be
treated similarly. Let $a,b\in L_{1}$ such that $a\leq b$. Assume first that
$\mathcal{W}_{a}\neq\emptyset$, and fix an arbitrary $\left(  x_{2}%
,\ldots,x_{n}\right)  \in\mathcal{W}_{a}$. Then $\varphi_{1}^{f}\left(
a\right)  =w_{a}=f\left(  a,x_{2},\ldots,x_{n}\right)  $, and since $f$ is
order-preserving, by the definition of $\mathcal{W}_{a}$, it follows that%
\[
f\left(  0,x_{2},\ldots,x_{n}\right)  <f\left(  a,x_{2},\ldots,x_{n}\right)
\leq f\left(  b,x_{2},\ldots,x_{n}\right)  \leq f\left(  1,x_{2},\ldots
,x_{n}\right)  .
\]
If $f\left(  b,x_{2},\ldots,x_{n}\right)  <f\left(  1,x_{2},\ldots
,x_{n}\right)  $ then $\left(  x_{2},\ldots,x_{n}\right)  \in\mathcal{W}_{b}$,
hence, by (\ref{eq consistency fif}) and (\ref{eq w}) we have $\varphi_{1}%
^{f}\left(  b\right)  =w_{b}=f\left(  b,x_{2},\ldots,x_{n}\right)  $. If
$f\left(  b,x_{2},\ldots,x_{n}\right)  =f\left(  1,x_{2},\ldots,x_{n}\right)
$, then $\left(  x_{2},\ldots,x_{n}\right)  \in\mathcal{L}_{b}$, therefore we
have $\varphi_{1}^{f}\left(  b\right)  \geq l_{b}\geq f\left(  b,x_{2}%
,\ldots,x_{n}\right)  $ by (\ref{eq consistency fif}) and (\ref{eq l}). In
both cases we obtain that%
\[
\varphi_{1}^{f}\left(  a\right)  =w_{a}=f\left(  a,x_{2},\ldots,x_{n}\right)
\leq f\left(  b,x_{2},\ldots,x_{n}\right)  \leq\varphi_{1}^{f}\left(
b\right)  ,
\]
since $f$ is order-preserving.

The case $\mathcal{W}_{b}\neq\emptyset$ can be treated similarly. Let us now
consider the remaining case $\mathcal{W}_{a}=\mathcal{W}_{b}=\emptyset$. Then%
\[
\mathcal{L}_{a}\cup\mathcal{U}_{a}=L_{2}\times\cdots\times L_{n}%
\setminus\mathcal{E}_{a}=L_{2}\times\cdots\times L_{n}\setminus\mathcal{E}%
_{b}=\mathcal{L}_{b}\cup\mathcal{U}_{b}.
\]
Furthermore, from $a\leq b$ we can conclude that $\mathcal{L}_{a}%
\subseteq\mathcal{L}_{b}$ and $\mathcal{U}_{a}\supseteq\mathcal{U}_{b}$ by
making use of the fact that $f$ is order-preserving. This implies that either
$\mathcal{L}_{a}\subset\mathcal{L}_{b}$ and $\mathcal{U}_{a}\supset
\mathcal{U}_{b}$, or $\mathcal{L}_{a}=\mathcal{L}_{b}$ and $\mathcal{U}%
_{a}=\mathcal{U}_{b}$. In the first case, by choosing an arbitrary $\left(
x_{2},\ldots,x_{n}\right)  \in\mathcal{L}_{b}\setminus\mathcal{L}%
_{a}=\mathcal{U}_{a}\setminus\mathcal{U}_{b}$ we obtain the desired inequality
with the help of (\ref{eq l}), (\ref{eq u}) and (\ref{eq consistency fif}):%
\[
\varphi_{1}^{f}\left(  a\right)  \leq u_{a}\leq f\left(  a,x_{2},\ldots
,x_{n}\right)  \leq f\left(  b,x_{2},\ldots,x_{n}\right)  \leq l_{b}%
\leq\varphi_{1}^{f}\left(  b\right)  \text{.}%
\]

In the second case, we claim that $f\left(  a,x_{2},\ldots,x_{n}\right)
=f\left(  b,x_{2},\ldots,x_{n}\right)  $ for all $\left(  x_{2},\ldots
,x_{n}\right)  \in L_{2}\times\cdots\times L_{n}$. This is clear if $\left(
x_{2},\ldots,x_{n}\right)  \in\mathcal{E}_{a}=\mathcal{E}_{b}$. If $\left(
x_{2},\ldots,x_{n}\right)  \in\mathcal{L}_{a}=\mathcal{L}_{b}$, then
\[
f\left(  a,x_{2},\ldots,x_{n}\right)  =f\left(  1,x_{2},\ldots,x_{n}\right)
=f\left(  b,x_{2},\ldots,x_{n}\right)  .
\]
If $\left(  x_{2},\ldots,x_{n}\right)  \in\mathcal{U}_{a}=\mathcal{U}_{b}$,
then
\[
f\left(  a,x_{2},\ldots,x_{n}\right)  =f\left(  0,x_{2},\ldots,x_{n}\right)
=f\left(  b,x_{2},\ldots,x_{n}\right)  .
\]
Thus $\mathcal{L}_{a}^{f}=\mathcal{L}_{b}^{f}$ and $\mathcal{U}_{a}%
^{f}=\mathcal{U}_{b}^{f}$, hence $l_{a}=l_{b}$ and $u_{a}=u_{b}$ (whenever
they are defined). Therefore $\varphi_{1}^{f}\left(  a\right)  $ and
$\varphi_{1}^{f}\left(  b\right)  $ coincide, no matter which rule (L),(U) or
(LU) was used to compute their values.
\end{proof}

\begin{remark}
\label{rem fi order preserving} We can see from the proof of the above lemma
that (LU) can be relaxed: $\varphi_{1}^{f}\left(  x_{1}\right)  $ could be
chosen to be any element of the interval $\left[  l_{x_{1}},u_{x_{1}}\right]
$ with the convention that whenever we encounter the same interval for
different values of $x_{1}$, we always choose the same element of this
interval. This guarantees that $\varphi_{1}^{f}$ will be order-preserving. The
proof of Lemma~\ref{lemma last} below works with this relaxed rule, since it
relies only on the fact that $\varphi_{1}^{f}\left(  x_{1}\right)  \in\left[
l_{x_{1}},u_{x_{1}}\right]  $ whenever $\varphi_{1}^{f}\left(  x_{1}\right)  $
is determined by rule the (LU).
\end{remark}

\begin{lemma}
\label{lemma last}Algorithm SUFF does not return the value
{\emph{\textbf{false}}} if and only if $f$ is an order-preserving
pseudo-Sugeno integral. In this case $f$ is pseudo-median decomposable w.r.t.
$\varphi_{1}^{f},\ldots,\varphi_{n}^{f}$.
\end{lemma}

\begin{proof}
For the sufficiency, let us suppose that $f\left(  \mathbf{x}\right)
=q\left(  \varphi_{1}\left(  x_{1}\right)  ,\ldots,\varphi_{n}\left(
x_{n}\right)  \right)  $ is an order-preserving pseudo-Sugeno integral, where
$q$ is a Sugeno integral and each $\varphi_{i}$ satisfies the boundary
conditions. Clearly, algorithm SUFF will not return \textbf{false} in
line~\ref{algline OP not OK}. Let us note that in the considerations of
Subsection~\ref{subsect algo} we did not make use of the fact that each
$\varphi_{k}$ is order-preserving, only the order-preservation of $f$, and the
pseudo-median decomposition was used. Since the latter holds for pseudo-Sugeno
integrals, the observations in Subsection~\ref{subsect algo} still hold for
$f$. In particular, (\ref{eq consistency fi}) holds for $f$, and this means
that the algorithm will not return \textbf{false} in
lines~\ref{algline W not OK} and \ref{algline l<w<u not OK}. Finally, since
$f$ is assumed to depend on all of its variables,
line~\ref{algline ess not OK} will not return \textbf{false} either.

For the necessity, let us assume that algorithm SUFF does not return
\textbf{false}. Then $f$ is clearly order-preserving, and by
Lemma~\ref{lemma fi order preserving}, the functions $\varphi_{1}^{f}%
,\ldots,\varphi_{n}^{f}$ are also order-preserving (hence they satisfy the
boundary conditions). By Theorem~\ref{thm:Med}, to prove that $f$ is a
pseudo-Sugeno integral it suffices to show that $f$ is pseudo-median
decomposable w.r.t. $\varphi_{1}^{f},\ldots,\varphi_{n}^{f}$. As in the proof
of the previous lemma, we focus on the first variable.

We need to show that%
\begin{equation}
\operatorname{med}\bigl(f\left(  0,x_{2},\ldots,x_{n}\right)  ,\varphi_{1}%
^{f}\left(  x_{1}\right)  ,f\left(  1,x_{2},\ldots,x_{n}\right)
\bigr)=f\left(  x_{1},x_{2},\ldots,x_{n}\right)  .\label{eq f'=f}%
\end{equation}
holds identically. We separate four cases with respect to the partition of
$L_{2}\times\cdots\times L_{n}=\mathcal{W}_{x_{1}}\cup\mathcal{L}_{x_{1}}%
\cup\mathcal{U}_{x_{1}}\cup\mathcal{E}_{x_{1}}$.

If $\left(  x_{2},\ldots,x_{n}\right)  \in\mathcal{W}_{x_{1}}$, then
$\varphi_{1}^{f}\left(  x_{1}\right)  =w_{x_{1}}=f\left(  x_{1},x_{2}%
,\ldots,x_{n}\right)  $ and%
\[
f\left(  0,x_{2},\ldots,x_{n}\right)  <f\left(  x_{1},x_{2},\ldots
,x_{n}\right)  <f\left(  1,x_{2},\ldots,x_{n}\right)  .
\]
Therefore $f\left(  0,x_{2},\ldots,x_{n}\right)  <\varphi_{1}^{f}\left(
x_{1}\right)  <f\left(  1,x_{2},\ldots,x_{n}\right)  $, hence the left hand
side of (\ref{eq f'=f}) is $\varphi_{1}^{f}\left(  x_{1}\right)  =f\left(
x_{1},x_{2},\ldots,x_{n}\right)  $.

If $\left(  x_{2},\ldots,x_{n}\right)  \in\mathcal{L}_{x_{1}}$, then
$\varphi_{1}^{f}\left(  x_{1}\right)  \geq l_{x_{1}}$ according to
(\ref{eq consistency fif}). Then by (\ref{eq l}) and by the definition of
$\mathcal{L}_{x_{1}}$ we get
\[
\varphi_{1}^{f}\left(  x_{1}\right)  \geq l_{x_{1}}\geq f\left(  x_{1}%
,x_{2},\ldots,x_{n}\right)  =f\left(  1,x_{2},\ldots,x_{n}\right)  .
\]
Therefore, both sides of (\ref{eq f'=f}) are equal to $f\left(  1,x_{2}%
,\ldots,x_{n}\right)  $.

The case $\left(  x_{2},\ldots,x_{n}\right)  \in\mathcal{U}_{x_{1}}$ follows
similarly. Finally, if $\left(  x_{2},\ldots,x_{n}\right)  \in\mathcal{E}%
_{x_{1}}$, then%
\[
f\left(  0,x_{2},\ldots,x_{n}\right)  =f\left(  x_{1},x_{2},\ldots
,x_{n}\right)  =f\left(  1,x_{2},\ldots,x_{n}\right)  ,
\]
hence (\ref{eq f'=f}) holds independently of the value of $\varphi_{1}%
^{f}\left(  x_{1}\right)  $.
\end{proof}

Lemmas~\ref{lemma fi order preserving} and \ref{lemma last}  together with
Theorem~\ref{thm dnf for pseudo-Sugeno} immediately yield
Theorem~\ref{thm alg correct}.

\section*{Acknowledgments.}

We would like to thank Jean-Luc Marichal for useful discussions and for
bringing this topic to our attention. The second named author acknowledges
that the present project is supported by the National Research Fund,
Luxembourg, and cofunded under the Marie Curie Actions of the European
Commission (FP7-COFUND), and supported by the Hungarian National Foundation
for Scientific Research under grant no. K77409.

\end{document}